\documentclass[12pt,letterpaper, final]{amsart}
\usepackage{amssymb, amsmath, amsthm}
\usepackage{mathrsfs}  
\usepackage{graphicx}
\usepackage[colorlinks=true, citecolor=blue, linkcolor=red, urlcolor=blue]{hyperref}
\usepackage[text={6in,8in},centering]{geometry}
\usepackage[dvipsnames]{xcolor}  

\usepackage{mathrsfs}  
\usepackage{verbatim}
\usepackage{mathrsfs}  

\usepackage{xcolor}
\usepackage{bm}

\usepackage{ifdraft}
\ifoptionfinal{
	\usepackage[disable]{todonotes}
}{
	\usepackage[norefs, nocites]{refcheck}
	
	\usepackage[notref, notcite]{showkeys}
	\usepackage[bordercolor=white, color=white]{todonotes}
}

\makeatletter\providecommand\@dotsep{5}\def\listtodoname{List of Todos}\def\listoftodos{\hypersetup{linkcolor=black}\@starttoc{tdo}\listtodoname\hypersetup{linkcolor=blue}}\makeatother

\usepackage{hyperref}
\usepackage{etoolbox}

\newtheorem{theorem}{Theorem}[section]
\newtheorem{lemma}[theorem]{Lemma}

\newtheorem{proposition}[theorem]{Proposition}
\theoremstyle{definition}

\theoremstyle{remark}
\newtheorem{remark}[theorem]{Remark}

\numberwithin{equation}{section}

\def\A{\mathcal A}

\def\D{\mathbb D}

\def\N{\mathbb N}

\def\O{\mathcal O}
\def\L{\mathcal L}

\renewcommand{\leq}{\leqslant}
\renewcommand{\geq}{\geqslant}

\newcommand{\norm}[1]{\left\|#1 \right\|}

\newcommand*\xbar[1]{%
	\hbox{%
		\vbox{%
			\hrule height 0.5pt 
			\kern0.5ex
			\hbox{%
				\ensuremath{#1}%
			}%
		}%
	}%
} 

\title[{\tiny  Calder\'{o}n problem for a logarithmic Schr\"{o}dinger operator on manifolds}]{Anisotropic Calder\'{o}n problem for a logarithmic Schr\"{o}dinger operator of order $2+$   on closed Riemannian manifolds}

\author[ ]{SAUMYAJIT DAS}
\address
{SAUMYAJIT DAS, Harish-Chandra Research Institute, A CI of Homi Bhabha National Institute, Chhatnag Road, Jhunsi, Allahabad 211 019, India }

\email{saumyajit.math.das@gmail.com}

\author[ ]{TUHIN GHOSH}
\address{TUHIN GHOSH,  Harish-Chandra Research Institute, A CI of Homi Bhabha National Institute, Chhatnag Road, Jhunsi, Allahabad 211 019, India}
\email{tuhinghosh@hri.res.in}

\author[]{Susovan Pramanik}
\address
{SUSOVAN PRAMANIK, Harish-Chandra Research Institute, A CI of Homi Bhabha National Institute, Chhatnag Road, Jhunsi, Allahabad 211 019, India}
\email{susovanpramanik@hri.res.in}

\begin{document}
	\maketitle
	\begin{abstract}  
		In this article, we study the anisotropic Calderón problems for the non local logarithimic Schr\"{o}dinger operators $(-\Delta_g+m) \log(-\Delta_{g}+m)+V$ with $m> 1$  on a closed, connected, smooth Riemannian manifold of dimension $n\geq2$. We will show that, for the operator $(-\Delta_g+m) \log(-\Delta_{g}+m)+V$, the recovery of both the Riemannian metric and the potential is possible from the Cauchy data, in the setting of a common underlying manifold with varying metrics. This result is unconditional. The last result can be extended to the case of setwise distinct manifolds also. In particular, we demonstrate that for setwise distinct manifolds, the Cauchy data associated with the operator $(-\Delta_g+m) \log(-\Delta_{g}+m)+V$, measured on a suitable non-empty open subset, uniquely determines the Riemannian manifold up to isometry and the potential up to an appropriate gauge transformation. This particular result is unconditional when the potential is supported entirely within the observation set. In the more general setting—where the potential may take nonzero values outside the observation set—specific geometric assumptions are required on both the observation set and the unknown region of the manifold. 
		
		\medskip
	\end{abstract}

	\section{Introduction}
	
	The study of non-local operators has been a great interest in mathematical studies in recent years. Non-local operators encompass a broad class of mappings: they arise in integro-differential equations (e.g., \cite{CS07}, \cite{ros2014dirichlet}), in fractional geometric flows (e.g., \cite{JX11}), in analysis of non-local boundary controllability (e.g., \cite{biccari2025boundary}), and in inverse-type problems such as the Calderón problem (e.g., \cite{ghosh2020calderon}). The Calder\'{o}n problem asks whether one can determine the electrical conductivity of a medium by making voltage and current measurements at the boundary. In the anisotropic case, the Calder\'{o}n problem has a geometric nature (cf.\cite{LU89},\cite{Sal13}); in particular, it investigates whether one can determine the differential structure and topological characteristics (such as homology, cohomology, and Betti numbers (\cite{belishev08}) by observing the solutions of the Schr\"odinger operator from any open set of the manifold. Additionally, it also investigates the unique recovery of the potential, up to a certain gauge transformation. Let $(M,g)$ be a smooth closed (compact and without boundary) and connected Riemannian manifold of dimension $n\geq 2$. Let $V\in C^{\infty}(M)$ be a smooth function, referred to as the potential. We fix a non-empty open set $\mathcal{O}\subset M$, such that $M\setminus \overline{\mathcal{O}}$ is non-empty. We refer the set $\mathcal{O}$ as the observation set. We assume that the metric and the potential, as well as certain data arising from the Schr\"odinger equation, are known in the observation set $\mathcal{O}$, while the metric, potential, and geometric structure remain inaccessible in  $M\setminus\mathcal{O}$. In particular, we refer to the known data in the observation set as the Cauchy data set.
	\newline
	We begin with the well-known and extensively studied anisotropic Calderón problems. The first of these concerns whether the geometry of a Riemannian manifold can be recovered by analyzing the associated Cauchy data set
	\begin{equation}\label{0pen no potential}
		\widetilde{\mathcal{C}}^\mathcal{O}_{M, g}= \{(u|_\mathcal{O}, (-\Delta_g)u|_\mathcal{O}) \mid u \in C^\infty(M), \, -\Delta_g u = 0 \,\text{in} \,\, M\setminus \overline{\mathcal{O}}\},
	\end{equation}
	i.e., it investigates the possibility of recovering the geometric information of the unknown region of the manifold by analyzing harmonic functions defined there, where these functions are known on the observation set. In other words, suppose the solution with respect to the source term $f$ is known inside the observation set, where the source term is vanishing outside the observation set, i.e.,
	\begin{align*}
		u|_{\O}\ \text{known}, \quad \mbox{where}\ u \ \mbox{satisfy}\ -\Delta_g u=f, \ f\in C_0^{\infty}(\O).
	\end{align*} 
	Here $C_0^{\infty}(\O)$ denotes all the smooth functions in $M$ vanishing outside $\O$. The problem then investigates the possibility of recovering the geometric characteristics of the unknown region $M\setminus\O$ from this source-to-solution data. Note that knowledge of the function 
	$u$ in the observation set is sufficient to determine the Cauchy data set, since the Laplace–Beltrami operator is local in nature.  The second problem addresses the recovery of both the potential and the geometric structure of the Riemannian manifold. More precisely, it examines whether knowledge of the Cauchy data set 
	\begin{equation}\label{0pen}
		\widetilde{\mathcal{C}}^\mathcal{O}_{M, g, V}= \{(u|_\mathcal{O}, (-\Delta_g)u|_\mathcal{O}) \mid u \in C^\infty(M), \, -\Delta_g u+ Vu = 0 \,\text{in} \,\, M\setminus \overline{\mathcal{O}}\}
	\end{equation}
	determines the isometry class of the manifold $(M, g)$ and the potential $V$ up to the corresponding gauge transformation. \((-\Delta_g) \) represents the positive Laplace-Beltrami operator on \( (M, g)\), given in the local coordinate via
	\[
	-\Delta_g=-\frac{1}{\sqrt{|g|}}\sum\limits_{j,k=1}^{n} \frac{\partial}{\partial x^j}\left( \sqrt{|g|}g^{jk} \frac{\partial}{\partial^k}\right), 
	\]
	where $(g^{jk})=(g_{jk})^{-1}$ and $|g|=det (g_{jk})$. The problem was first described and analyzed by Calder\'on  in \cite{Calderon1980InverseBoundary}. The problem is solved in dim$(M)$=2 [see \cite{Nach96}, \cite{KL06}, \cite{Buk08}, \cite{GT11}, \cite{imanuvilov2011determination}, \cite{imanuvilov2012partial}, \cite{LU01}]. However, due to the conformal invariance of the Laplace–Beltrami operator in two dimensions, an additional gauge is required to solve the problem. For dimensions three and higher, the problem has been addressed under the assumption that the manifold is real analytic (see the references \cite{LU89}, \cite{LU01}, \cite{LTU03}), but it is still open for smooth manifolds. The extension looks promising due to the positive observation in transversally anisotropic geometries, see \cite{DKSU09}, \cite{DKLS16}. For a comprehensive survey for the Calder\'on problem, we refer the readers to \cite{SU87}, \cite{KSU07}, \cite{Uhl14}. 
	\newline
	The Laplacian can be viewed as the limiting operator of fractional Laplacian $(-\Delta_g)^s$ for $s\in(0,1)$. In Euclidean domain $\mathbb{R}^n$, fractional Laplacian associated to jump L\'evy's process. One can think of recover the Laplacian as $s\to 1$ and the identity map as $s\to 0$. The linearized deviation of $(-\Delta_g)^s$ from the identity and the Laplacian is given by:
	\[
	\frac{d}{ds} (-\Delta_g)^s\Big|_{s=0}= \log {(-\Delta_g)}, \  \ \ \frac{d}{ds} (-\Delta_g)^s\Big|_{s=1}= \log {(-\Delta_g)}=(-\Delta_{g})\log{(-\Delta_g)}, \ \text{respectively}.
	\] 
	Hence, in the asymptotic expansion of the fractional Laplacian, the operators 
	\[
	\log\bigl(-\Delta_g\bigr)
	\quad\text{and}\quad
	(-\Delta_g)\,\log\bigl(-\Delta_g\bigr)
	\]
	appear respectively around \(s\!\sim\!0\) and \(s\!\sim\!1\). We refer the reader to the articles \cite{chen2025logarithmic} and \cite{chen2019dirichlet} for a comprehensive study of the operator 
	\(\log\bigl(-\Delta_g\bigr)\), which is an operator of order \(0+\). In this article we focus on the asymptotic behaviour near \(s = 1\), i.e., the operator $(-\Delta_g)\,\log\bigl(-\Delta_g\bigr)$, which is a non local operator of order $2+$.  Note that the order of this non local operator is almost same as order of the Laplace-Beltrami operator. However, for well-posedness of this operator via functional calculus, instead of $
	(-\Delta_g)\,\log\!\bigl(-\Delta_g\bigr)
	$ we consider the operator $
	(-\Delta_g + m)\,\log\!\bigl(-\Delta_g + m\bigr)$ with \(m>1\) a constant. The eigenvalues of the operator $
	(-\Delta_g + m)\,\log\!\bigl(-\Delta_g + m\bigr)$ are strictly greater than \(1\); hence the logarithmic part is always positive. This non local operator can be viewed as the near $s\sim 1$ asymptotes of the fractional relativistic operator $(-\Delta_g + m)^s$, where $s\in(0,1)$. In particular:
	\[
	\frac{d}{ds} (-\Delta_g + m)^s\Big|_{s=1}= (-\Delta_g + m)\,\log\!\bigl(-\Delta_g + m\bigr).
	\]
	Before moving to the Calderón problem for this particular logarithmic-relativistic operator, we begin by describing the fractional version of the Calderón problem, non local in nature. The problem begins with the recovery of the geometric information of the Riemannian manifold from the Cauchy data set
	\begin{equation}\label{0pen fractional no potential}
		\widetilde{\mathcal{C}}^\mathcal{O}_{M, g}= \{(u|_\mathcal{O}, (-\Delta_g)^{\alpha}u|_\mathcal{O}) \mid u \in C^\infty(M), \, (-\Delta_g)^{\alpha} u= 0 \,\text{in} \,\, M\setminus \overline{\mathcal{O}}\}.
	\end{equation}
	This Cauchy data set can be interpreted analogously to the source to solution data for the fractional Laplacian, where the source consists of smooth functions vanishing outside the observation set, and one observes both the solution and the action of the fractional Laplacian on it within the observation region. Due to the non local nature of the fractional Laplacian, here we need to know both the solution and the action of the fractional Laplacian on it within the observation set. The problem can be generalised to find both the potential and the isometry class of the Riemanian manifold $(M,g)$. Here, the question is whether the following Cauchy data set
	\begin{equation}\label{0pen fractional}
		\widetilde{\mathcal{C}}^\mathcal{O}_{M, g, V}= \{(u|_\mathcal{O}, (-\Delta_g)^{\alpha}u|_\mathcal{O}) \mid u \in C^\infty(M), \, (-\Delta_g)^{\alpha} u+ Vu = 0 \,\text{in} \,\, M\setminus \overline{\mathcal{O}}\}
	\end{equation}
	determines the topological and differential aspects of the manifold and the potential. Here, the fractional exponent $\alpha$ is assumed to lie in the interval $(0,1)$. When the potential $V=0$, the problem is solved in \cite{Fei24}, \cite{FGKU25}, \cite{ruland2023revisiting}. Then method with zero potential has been extended to non compact manifold also in \cite{CO24}. In \cite{FKU24} authors recover both the Riemannian metric and the potential. Although in \cite{FKU24}, authors needs two important assumptions to recover both the metric and the potential. One of the assumption is that the potential should vanish inside the observation set and the other is the unknown portion of the Riemannian manifold $M\setminus\mathcal{O}$ is nontrapping and atleast one antipodal set corresponding to a point in the observation set should contained inside the observation set $\mathcal{O}$. This additional geometric assumption guarantees that the geodesic flow transmits geometric information from the unknown region of the manifold to the observation set.  
	\newline

	\begin{remark}
		The study of non-local inverse problems has been motivated from various physical phenomenon described in \cite{AMRT10}, \cite{GO08}, \cite{Zim23}, \cite{lin2024uniqueness}, \cite{BV16}. The study of non-local inverse problems and boundary control also occupies a prominent place in mathematical research.  For the sake of the readers We like to mention some references here:   \cite{Fei24}, \cite{GSU20}, \cite{GLX17}, \cite{RS20}, \cite{RS2020}, \cite{GRSU20}, \cite{MLR20}, \cite{Cov20}, \cite{Li20}, \cite{HL20}, \cite{BGU21}, \cite{KLW22}, \cite{CGR22}, \cite{Gho22}, \cite{Zim23}, \cite{HU24}, \cite{HLW24}, \cite{Das25}. However, the operator's order in consideration falls within $(0, 2)$.  We like to emphasize the fact that in this article the order of non local operator is $2+$. For general definition of order of the operators we refer the readers to \cite{GS94}, \cite{Tay81}, \cite{Tay23}.  
	\end{remark}
	
	Recently, non-local Schrödinger operators of order greater than $2$ have attracted significant interest in the mathematical community. One of such operator is $(-\Delta_g+m)\log{(-\Delta_g+m)}$ with $m>0$ whose order is $2+$. In \cite{pramanik2025anisotropic}, author analyzed the Calder\'on problem with $V\equiv 0$  and recovered the Riemannian metric with the help of heat flow. In a different setting, the authors in \cite{HLW24} recently studied inverse problems concerning the recovery of a potential associated with the near-zero logarithmic operator, modelled through Schr\"{o}dinger-type equations in flat spaces. In this article, we generalize the result by including a potential supported on the observation set $\mathcal{O}$, and then extend the same result to arbitrary smooth potentials under the geometric assumptions described in \cite{FKU24}. Mostly our calculations  follow the ideas outlined in \cite{FKU24}. Below, we describe the two questions in more precise mathematical terms.

	\textbf{Recovery of the geometry of the Riemannian manifold:}
	\newline
	This problem can be formulated via the source-to-solution data within the observation set, analogous to the case of the fractional Laplacian or the complex power of Laplacian. Here, the operator is the composition of Laplacian with logarithmic Laplacian. The question is: whether the following Cauchy data set
	\begin{equation}
		\label{cauchy_data, smooth, only metric}
		\mathcal{C}_{M,g,V}^O =
		\left\{
		\begin{aligned}
			(u|_O, ((-\Delta_g+m) \log{(-\Delta_g
				+m)} u)|_O) \,:\, u \in C^\infty(M),\\
			(-\Delta_g+m)\log{(-\Delta_g+m)}u+Vu = 0 \ \text{on} \ M \setminus \overline{O}, \, V\in C_0^{\infty}(\O).
		\end{aligned}
		\right \}
	\end{equation}
	is sufficient to determine the topological and differential structure of the Riemannian manifold. More precisely, we have the following result:
	
	\begin{theorem}
		\label{thm_potential_with compact_supp}
		Let $m> 1 $. For $j=1,2$, let $(M_j,g_j)$ be a smooth closed and connected Riemannian manifold of dimension $n \geq 2$. Let $\mathcal{O}\subset M_1\cap M_2$ be a nonempty open connected set such that $M_j\setminus\overline{\mathcal{O}}\ne \emptyset$, $j=1,2$,   and assume that for $j=1,2$, $(\mathcal{O},g_1)=(\mathcal{O},g_2):=(\mathcal{O},g)$. Let $V \in C_0^{\infty}(\O)$ such that  zero is not an eigenvalue of the operators $(-\Delta_{g_j}+m)\log{(-\Delta_{g_j}+m)}+V$, defined in $M_j$, for $j=1,2$.  Then the equality of the Cauchy data set over $(\mathcal{O}, g)$,  $$i.e., \quad \mathcal C_{M_1,g_1,V}^{\mathcal{O}} = \mathcal C_{M_2,g_2,V}^{\mathcal{O}}\,,$$ implies the existence of a diffeomorphism $\Phi: M_1 \to M_2$ such that $\Phi^\ast g_2 = g_1$.
	\end{theorem}
	Here, when the potential is supported within the observation set, the heat flow carries geometric information from the unknown part of the manifold. Therefore, no additional conditions need to be imposed on the unknown region of the manifold. 
	
	\begin{remark}
		The case where the potential vanishes identically, i.e., $V\equiv 0$, follows as a straightforward corollary of the above result. This implies that, in the absence of a potential, the Cauchy data set alone suffices to recover the geometric information of the Riemannian manifold. Hence we obtain the result described in \cite{pramanik2025anisotropic} as a corollary of the above theorem.
	\end{remark}
	
	\textbf{Reovery of the geometry of Riemannian manifold and the potential:}
	\newline
	The second problem investigate the recovery of the potential too along with the isometry class of the Riemannian manifold. We examine, whether the following Cauchy data set
	\begin{equation}\label{cauchy_data, smooth}
		\mathcal{C}_{M,g,V}^O =
		\left\{
		\begin{aligned}
			(u|_O, ((-\Delta_g+m) \log{(-\Delta_g+m)} u)|_O) \,:\, u \in C^\infty(M),\\
			(-\Delta_g+m)\log{(-\Delta_g+m)}u+Vu = 0 \ \text{on} \ M \setminus \overline{O}.
		\end{aligned}
		\right \}
	\end{equation}
	is sufficient to recover both the geometric aspects of the Manifold and the potential. While our study largely follows the techniques developed in  \cite{FKU24}, we emphasize once again that the assumptions required in \cite{FKU24}—particularly that concerning the vanishing of the potential—are not needed in our approach. The problem is categorized into the following two types.
	\begin{itemize}
		\item [\textbf{A1})] In this case where $M_1$ and $M_2$ coincide as sets, i.e., $M_1=M_2$ setwise.
		\item [\textbf{A2})] In this case $M_1\neq M_2$ but the observation set $\O$ is contained in $M_1\cap M_2$.
	\end{itemize}
	In the first case, we are able to recover information about both the geometry of the manifold and the potential. The precise result established in this article is as follows:
	\begin{theorem}
		\label{recovering_metric_and_potential_from_one_manifold}
		Let $m> 1 $. For $j=1,2$, let $(M,g_j)$ be a smooth closed and connected Riemannian manifold of dimension $n \geq 2$, and let $V_j \in C^{\infty}(M)$. Further assume zero is not an eigenvalue of the operators $\left( (-\Delta_{g_j}+m)\log{(-\Delta_{g_j}+m)} +V_j \right)$, defined in $(M,g_j)$, for $j=1,2$. Let $\mathcal{O}\subset M$ be a nonempty open connected set such that $M\setminus\overline{\mathcal{O}}\ne \emptyset$   and assume that  $(\mathcal{O},g_1)=(\mathcal{O},g_2):=(\mathcal{O},g)$ and $V_1|_{\mathcal{O}}=V_2|_{\mathcal{O}}$. 
		\newline 
		Then the equality of the Cauchy data set over $(\mathcal{O}, g)$,  $$i.e. \quad \mathcal C_{M,g_1,V_1}^{\mathcal{O}} = \mathcal C_{M,g_2,V_2}^{\mathcal{O}}\,,$$ implies the existence of a diffeomorphism $\Phi: M \to M$ such that $\Phi^\ast g_2 = g_1$ and $V_1= V_2\circ\Phi$. Moreover, the smooth diffeomorphism $\Phi$ is equal to identity on the set $\mathcal{O}$.
	\end{theorem}
	In this case, we do not require any specific assumptions on the observation set, other than that it is open and that both the set and its complement in $M$ are non-empty. The geometry of the manifold and the potential can be recovered either by analyzing the heat flow through the observation set or, in an equivalent way, by studying the spectral data restricted to the observation set.
	
	\vspace{.1cm}
	
	However, for the second case $\textbf{A2}$, where the manifolds are setwise distinct, we do require a geometric condition on the unknown region of the manifold—specifically, the nontrapping condition—as well as the assumption that at least one of the antipodal sets lies inside the observation set, both of which are essential to establishing our result as outlined in \cite{FKU24}. The geometric assumptions—namely, the nontrapping condition and the antipodal set—are described in detail below (see \cite{FKU24}).
	\newline\textbf{Nontrapping:} 
	let us denote the unit tangent bundle of the manifold $M \setminus O$ by $S(M \setminus O)$, and for any $(x, v) \in S(M \setminus O)$, we denote by $\gamma_{x,v}$ the unique geodesic on $M \setminus O$ such that $\gamma_{x,v}(0) = x$ and $\dot{\gamma}_{x,v}(0) = v$. Here, we view $M \setminus O$ as a compact manifold with boundary embedded in the closed manifold $M$, and the geodesic vector field is defined on the whole of $M$. We consider the forward and backward exit time functions, defined by
	\begin{align*}
		&\tau_+: S(M \setminus O) \to [0, +\infty], \ \tau_+(x,v) = \sup\{t \ge 0: \gamma_{x,v}(s) \in M \setminus O\, \, \forall s \in [0, t]\},\\
		&\tau_-: S(M \setminus O) \to [-\infty, 0], \ \tau_-(x,v) = -\sup\{t \ge 0: \gamma_{x,v}(s) \in M \setminus O\, \,  \forall s \in [-t, 0]\}.
	\end{align*}
	The fact that the manifold $M \setminus O$ is non-trapping means that $\tau_+(x,v) < +\infty$ and $-\tau_-(x,v) < +\infty$ for all $(x,v) \in S(M \setminus O)$.
	\newline
	\textbf{Antipodal set:} Given any $p\in M$, we define the antipodal set of $p$, denoted by $\mathcal A_{M,g}(p)$, as the set of all points that are antipodal to $p$, that is to say,
	$$ \mathcal A_{M,g}(p) = \{q \in M\,:\, \textrm{dist}_g(p,q)= \max_{p'\in M} \textrm{dist}_{g}(p,p') \}.$$
	Note that since the manifold $M$ is compact, for each $p \in M$, we have $\mathcal{A}_{M,g}(p) \ne \emptyset$. For further details we refer the readers to \cite{FKU24}, \cite{chen1988riemannian}, \cite{tanaka2012intersection}. We now proceed to state another main result of this article, established under a specific geometric assumption on both the observation set and the unknown region of the Riemannian manifold.  
	\begin{theorem}\label{update_with_smooth_potential}
		Let $m> 1 $. For $j=1,2$, let $(M_j,g_j)$ be a smooth closed and connected Riemannian manifold of dimension $n \geq 2$, and let $V_j \in C^{\infty}(M_j)$. Further assume zero is not an eigenvalue of the operators $\left( (-\Delta_{g_j}+m)\log{(-\Delta_{g_j}+m)} +V_j \right)$, defined in $M_j$, for $j=1,2$. Let $\mathcal{O}\subset M_1\cap M_2$ be a nonempty open connected set such that $M_j\setminus\overline{\mathcal{O}}\ne \emptyset$, $j=1,2$,   and assume that for $j=1,2$, $(\mathcal{O},g_1)=(\mathcal{O},g_2):=(\mathcal{O},g)$ and $V_1|_{\mathcal{O}}=V_2|_{\mathcal{O}}$. 
		
		\begin{itemize}
			\item [{\bf(H)}]{$(M_j\setminus O,g_j)$ is nontrapping and there exists $p\in O$ such that $\mathcal A_{M_j,g_j}(p)\subset O$.}
		\end{itemize}
		Then the equality of the Cauchy data set over $(\mathcal{O}, g)$,  $$i.e. \quad \mathcal C_{M_1,g_1,V_1}^{\mathcal{O}} = \mathcal C_{M_2,g_2,V_2}^{\mathcal{O}}\,,$$ implies the existence of a diffeomorphism $\Phi: M_1 \to M_2$ such that $\Phi^\ast g_2 = g_1$ and $V_1= V_2\circ\Phi$. Moreover, the smooth diffeomorphism $\Phi$ is equal to identity on the set $\mathcal{O}$.
	\end{theorem}
	
	\begin{remark}
		The previous theorem also extends to the case where only the potential is recovered on a given manifold, that is, when $(M_1,g_1)\equiv (M_2,g_2)\equiv (M,g)$. In this case, the isometric diffeomorphism is the identity map, so the problem reduces to recovering only the potential from the Cauchy data. It follows directly from the unique continuation property  as described in Theorem \ref{thm_ucp}. Hence, in particular we don't need any geometric assumption on the observation set and the unknown region. Therefore, for the recovery of the potential from the Cauchy data on a given Riemannian manifold, all geometric assumptions in Theorem~\ref{update_with_smooth_potential} can be omitted.
		
	\end{remark}
	
	We transform the Calder\'on problem to the Gel'fand inverse spectral problem via unique continuation principle as presented in \cite{FKU24}. Equality of Gel'fand spectra along with the nontrapping condition and vanishing of potential implies the diffeomorphism of Riemannian manifolds (see \cite[Theorem 1.11]{FKU24}). It also determines the potential up to some gauge transformation \cite{FKU24}. Below we describe variant of Gel'fand inverse spectral problem.
	
	\subsection{Variant of Gel'fand inverse spectral problem}
	
	The recovery of the geometry or the potential from the spectral data has long been of great interest in the mathematical community. In many cases, the spectral data on a manifold carries significant geometric information. In inverse spectral theory this spectral data within observation set called as Gel'fand spectral data. Few notable foundational works in this inverse spectral problem field can be found in \cite{borg1946umkehrung}, \cite{gel1951determination}, \cite{gel1954some}, \cite{levinson1949inverse}.  It starts with the recovery of potential from the knowledge of the Dirichlet
	eigenvalues and the boundary traces of the normal derivatives of the normalized
	Dirichlet eigenfunctions for the bounded Euclidean domain [see \cite{Nach96}, \cite{novikov1988multidimensional}]. The normal derivative denotes the measure of boundary flux and one can measure it from Gel'fand spectral data. It is extended for the fractional psedudifferential operator in Euclidean domain in \cite{das2025fractional}.
	\newline
	In Riemannian manifold (with or without boundary), the problem becomes the recovery of the geometric or topological information from Gel'fand spectral data. Here, the problem asks whether the Gel'fand spectral data uniquely determines both the topological and differential characteristics of the Riemannian manifold, as well as the potential, up to a gauge transformation. For the closed Riemannian manifold the Gel'fand data set is the eigenvalues and the orthonormal eigenfunctions in the observation set $\mathcal{O}$.  For local case, we refer the readers to \cite{bosi2022reconstruction}, \cite{HLOS18}, \cite{krupchyk2008inverse}. We also refer the readers to \cite{kurylev2012uniqueness}, \cite{HLOS18},  \cite{helin2020inverse}, \cite{fefferman2021reconstruction}, \cite{anderson2004boundary} for
	various works in the context of inverse  spectral problems for both parabolic and hyperbolic equations. In Euclidean domain, the equivalence of Gel'fand spectral problem for heat, wave or Schr\"odinger operator has shown in \cite{katchalov2004equivalence}. The Gelfand spectral problem also have a distinguished place Boundary control method. For additional background and details, we direct the readers to \cite{belishev1987approach}, \cite{belishev2007recent}, \cite{belishev1992reconstruction}, \cite{tataru1995unique}. 
	\newline
	We transform our Calder\'on problem into Gel'fand spectral problem. In the case where the manifolds coincide setwise (\textbf{A1}), For general smooth potential $V$, the geometric information can be readily recovered using the following theorem. 
	\begin{theorem}[\cite{HLOS18} \label{Helin}]
		Let $(M,g)$ be a smooth and compact Riemannian manifold without boundary. Let $\O\subset M$ be an open non empty set. Let $(\phi_k)_{k\in\mathbb{N}}\subset C^{\infty}(M)$ be the collection of orthonormal eigenfunctions of the operator $\Delta_g$ in $L^2(M)$. Let $(\lambda_k)_{k\in\mathbb{N}}$ be the collection of corresponding eigenvalues of $\Delta_g$. Then the spectral dat
		\[
		\Big(\O, (\lambda_k)_{k\in\mathbb{N}}, (\phi_k|_{\O})_{k\in\mathbb{N}}\Big)
		\]
		determines $(M,g)$ upto isometry.
	\end{theorem}
	In the general settings, where the Riemannian manifolds are setwise distinct (\textbf{A2}), for general smooth potential $V$, the differential and topological structure of the Riemannian manifold and the recovery of the metric and the potential is an easy consequence of \cite[Theorem 1.11]{FKU24}, provided we assume the nontrapping condition on the unknown part of the manifold and vanishing of potential inside the observation set. However, if we consider the potential $V$ has a compact support inside the observation set $\mathcal{O}$, then we can identify the heat kernel uniquely inside the observation set and the following theorem recovers the differential and topological character of the manifold along with the Riemannian metric.
	\begin{theorem}[\cite{FGKU25}]\label{hear_ker_diffeo}
		Let $(M_1, g_1)$ and $(M_2, g_2)$ be smooth connected
		complete Riemannian manifolds of dimension $n\geq 2$ without boundary. Let $\mathcal{O}\subset M_1\cap M_2$, be a non empty open sets.  Assume furthermore that
		$$P_{{g_1}}(t, x, y) = P_{g_2}(t, x, y),\quad\forall \, t>0,\,\,\mbox{and }x,y\in\mathcal{O}.$$
		Then there exists a diffeomorphism $\varphi : M_1 \mapsto M_2$ such that $\varphi^\ast g_2 =g_1$ on $M_1$.
	\end{theorem} 
	The above analysis highlights the importance of reformulating the Calder\'on problem in terms of Gel'fand spectral data. The next theorem determines the differential and topological characteristic of the Riemannian manifold and recovers the metric under some geometric condition on the manifold \cite{FKU24}.
	\begin{theorem}[\cite{FKU24}]
		\label{spec_diffeo}
		For $j=1,2$, let $(M_j,g_j)$ be a smooth closed and connected Riemannian manifold of dimension $n \geq 2$. Let $O \subset M_1 \cap M_2$ be a nonempty open set such that $M_j\setminus\overline{O}\ne \emptyset$, and assume that the condition (H) as in the Theorem \ref{t1}, is satisfied for $j=1,2$. Assume also that $g_1|_{O} = g_2|_{O}$. For $j=1,2$, suppose that there exists an $L^2(M_j)$ Schauder basis consisting of eigenfunctions $\{\psi_k^{(j)}\}_{k=0}^{\infty} \subset C^{\infty}(M_j)$ for $-\Delta_{g_j}$ on $(M_j,g_j)$ corresponding to (not necessarily distinct) eigenvalues 
		\[0 = \mu_0^{(j)} < \mu_1^{(j)} \leq \mu_2^{(j)} \leq \mu_3^{(j)} \leq \ldots\]
		such that, given any $k=0,1,2,\ldots$, there holds
		\begin{equation}\label{psi_eq} \mu_k^{(1)} = \mu_k^{(2)} \quad \text{and} \quad \psi_k^{(1)}(x) = \psi_k^{(2)}(x) \quad \forall \, x \in O. \end{equation} 
		Then, there exists a smooth diffeomorphism $\Phi: M_1 \to M_2$ that is identity on $O$ such that $g_1 = \Phi^\star g_2$ on $M_1$.
	\end{theorem}
	One of the most important intermediate step to convert the Calder\'on problem to the Gel'fand spectral problem is unique continuation principle. For real powers of Laplacian we like to refer the readers to \cite{GSU20}, \cite{FKU24}. In this article, we generalize the unique continuation principle for the logarithimic Schr\"odinger operator $(-\Delta_g+m)\log (-\Delta_g+m)$. More precisely, we establish the following unique continuation principle:
	
	\begin{theorem}
		\label{thm_ucp}
		Let $m> 1$. Let $(M,g)$ be a smooth closed and connected Riemannian manifold of dimension $n \geq 2$, and let $\mathcal{O} \subset M$ be a nonempty open set. Let  $v \in C^\infty(M)$ satisfy
		\begin{equation}\label{t2_a_1} 
			v|_\mathcal{O} =0 \quad \text{and} \quad ((-\Delta_{g}+m)\log{(-\Delta_g+m)}) v|_\mathcal{O} = 0.
		\end{equation}
		Then $v \equiv 0$ on $M$.
	\end{theorem}

	\subsection{Organization of the paper}
	
	We divide the article in four sections. In the first section, we address the well-posedness and regularity of solutions to the nonlocal logarithmic Schr\"odinger equation with a non local operator $(-\Delta_g+m)\log(-\Delta_g+m)$ with $m> 1$ and a source term on the right-hand side. All the classical definitions, results, and the functional calculus approach presented in the first section can be found in \cite{Stri83}, \cite{str83}, and \cite{Tay81}.
	
	\vspace{.1cm}
	
	The second section is dedicated to the proof of the unique continuation principle, as stated in Theorem \ref{thm_ucp}. It is an important theorem joining the bridge between the Calder\'on problem and the Gel'fand problem.
	
	\vspace{.1cm}
	
	The third section explores the transformation of the Calderón problem into an inverse spectral problem. To carry out the transformation using the unique continuation principle, it is necessary to study the combined action of the heat kernel and the nonlocal operator Laplacian compose with it's logarithm on the solution of the nonlocal logarithimic Schr\"odinger equation within the observation set. In our case, it suffices to examine the composite action on solutions of the nonlocal logarithmic Schr\"odinger equation whose source term is smooth and compactly supported within the observation set $\mathcal{O}$.  Here, we also remind the readers that in our case, the non local operator is $(-\Delta_{g_j}+m)\log(-\Delta_{g_j}+m)$. This composite action is described in details in the Proposition \ref{key_prop_for_all_thm}. It allows us to transform the problem into Gel'fand spectral problem as described in Proposition \ref{Spectural data equal prop}.
	
	\vspace{.1cm}

	In the final two sections, we devote our attention to establish Theorem \ref{thm_potential_with compact_supp}, Theorem \ref{recovering_metric_and_potential_from_one_manifold} and Theorem \ref{update_with_smooth_potential}. The Gel'fand inverse spectral data, combined with Theorem \ref{hear_ker_diffeo}, enables the derivation of Theorem \ref{thm_potential_with compact_supp}, while the same combined with Theorem \ref{Helin}, directly yields Theorem \ref{recovering_metric_and_potential_from_one_manifold}. Furthermore, the geometric properties of the Riemannian manifold and the recovery of the metric in the case of Theorem \ref{update_with_smooth_potential} follow directly from \cite[Theorem 1.11]{FKU24}. Therefore, in the case of Theorem \ref{recovering_metric_and_potential_from_one_manifold} and Theorem \ref{update_with_smooth_potential}, our primary focus is on the recovery of the potential for the non local operator $(-\Delta_g+m)\log{(-\Delta_g+m)}$. Most of the ideas used here are taken from \cite{FKU24}.

	\vspace{.2cm}
	
	\textbf{Notation:}
	\begin{itemize}
		\item [$\bullet$] $\A_g:= -\Delta_g+m, \ m> 1$.\\
		\item [$\bullet$] $\L_g:= \A_g\log{\A_g}$ 
	\end{itemize}
	
	\section{Preliminaries}\label{sec_pre}
	
	Let $(M, g)$ be a smooth, closed and connected Riemannian manifold of dimension $n \geq 2$. Consider the positive Laplace--Beltrami operator $-\Delta_g$ on $M$, which is self-adjoint on $L^2(M)$ with domain $\mathcal{D}(-\Delta_g) = H^2(M)$, the standard Sobolev space on $M$. Denote the distinct eigenvalues of $-\Delta_g$ by
	\[
	0 = \lambda_0 < \lambda_1 < \lambda_2 < \ldots+\infty,
	\]
	and let $d_k$ denote the multiplicity of the eigenvalue $\lambda_k$ for $k = 0, 1, 2, \ldots$. For each $k$, let $\{ \phi_{k,\ell} \}_{\ell=1}^{d_k}$ be an $L^2(M)$-orthonormal basis for the eigenspace $\operatorname{Ker}(-\Delta_g - \lambda_k)$ corresponding to $\lambda_k$. The collection $\{ \phi_{k,\ell} \}_{k \geq 0,\, 1 \leq \ell \leq d_k}$ forms a complete orthonormal basis for $L^2(M)$. For each $k$, define the orthogonal projection operator $\pi_k : L^2(M) \to \operatorname{Ker}(-\Delta_g - \lambda_k)$ by
	\begin{equation}
		\label{proj_op}
		\pi_k f = \sum_{\ell=1}^{d_k} \langle f, \phi_{k,\ell} \rangle_{L^2(M)} \, \phi_{k,\ell}, \quad f \in L^2(M),
	\end{equation}
	where $\langle \cdot, \cdot \rangle_{L^2(M)}$ denotes the $L^2$ inner product on $M$.
	
	Motivated by the fact that for all $s > 0$, $\lim_{\lambda \to \infty} \frac{\log \lambda}{\lambda} = 0$, consider the (unbounded) operator $\mathcal{L}_g := (-\Delta_g + m)\log((-\Delta_g)+m)$
	on $L^2(M)$. The purpose of perturbing with $m\geq 1$ is to ensure that the operator is elliptic. By spectral theory, for $u \in L^2(M)$, we have the expansion
	\begin{align}\label{Laplace_log_laplace_spec_def}
		\mathcal{L}_g u = \sum_{k=0}^\infty (\lambda_k + m) \log (\lambda_k + m)\, \pi_k u,
	\end{align}
	where $\{ \lambda_k \}_{k=0}^\infty$ are the eigenvalues of $-\Delta_g$, $\{ \phi_{k, \ell} \}$ is an orthonormal basis of corresponding eigenfunctions, and $\pi_k u =\sum_{\ell=1}^{d_k} \langle u, \phi_{k,\ell} \rangle_{L^2(M)} \, \phi_{k,\ell},$ is the orthogonal projection onto the eigenspace associated to $\lambda_k$.
	The domain of $\mathcal{L}_g$ is
	\begin{align}\label{dom_of_lap_log_lap}
		D(\mathcal{L}_g) = \left\{ u \in L^2(M) \,\middle|\,
		\sum_{k=0}^{\infty}  \left[(\lambda_k + m)\log(\lambda_k + m)\right]^2 
		| \pi_{k }u|^2 < \infty \right\}:=\mathbb{H}(M).
	\end{align}
	Since for any $s>0$,\,\,$|\log(\lambda_k + m)|^2 <|(\lambda_k + m)|^{2s}$, it follows that  $H^{2+s}(M) \subset \mathbb{H}(M)$. In this context, we define the log-Laplacian operator \(\log(-\Delta_g + m \mathbb{I})\) by
	\begin{align}\label{log_laplace_spec_def}
		\log (-\Delta_g + m\mathbb{I})\,u= \sum_{k=0}^\infty\log (\lambda_k + m)\, \pi_k u,
	\end{align}
	with domain 
	\begin{align}\label{dom_of_log_lap}
		D(\log(-\Delta_g + m\mathbb{I})) = \left\{ u \in L^2(M) \,\middle|\,
		\sum_{k=0}^{\infty}   \log(\lambda_k + m)^2 
		| \pi_{k }u|^2 < \infty \right\},
	\end{align}
	thanks to the relation \( |\log(\lambda_k + m)|^2 < |(\lambda_k + m)|^{2s} \), we also have \( H^{2s}(M) \subset D(\log(-\Delta_g + m \mathbb{I})) \) for any \( s > 0 \).
	\medskip

	Let \(f \in L^2(M)\) and consider the direct problem
	\begin{equation}\label{equation_u}
		\mathcal{L}_g u = f \quad \text{in } M.
	\end{equation}
	Let us first observe that \(0\) is not an eigenvalue of the operator \(\mathcal{L}_g\) defined in \eqref{Laplace_log_laplace_spec_def} for \(m > 1\). Suppose, for contradiction, that \(0\) is an eigenvalue of \(\mathcal{L}_g\). Then there exists a nonzero function \(u \in \mathbb{H}(M)\) such that
	\[
	\langle u,\, \mathcal{L}_g u \rangle = 0.
	\]
	Expanding \(u\) in terms of the orthonormal \(L^2(M)\) eigenbasis \(\{\phi_k\}_{k=0}^\infty\), i.e., \(u = \sum_{k=0}^\infty \langle \phi_k, u \rangle \phi_k\), we obtain
	\[
	\left\langle \sum_{k=0}^\infty \langle \phi_k, u \rangle \phi_k,\ \sum_{k=0}^\infty (\lambda_k + m)\ln(\lambda_k + m)\, \langle \phi_k, u \rangle \phi_k \right\rangle = 0,
	\]
	which gives
	\[
	\sum_{k=0}^\infty |\langle \phi_k, u \rangle|^2\, (\lambda_k + m)\ln(\lambda_k + m) = 0.
	\]
	Since each term in the sum is non-negative (because \(\lambda_k \geq 0\) for all \(k\) and \(m > 1\)), it follows that
	\[
	\langle \phi_k, u \rangle = 0 \quad \text{for all } k \geq 0,
	\]
	which implies \(u = 0\), contradicting our assumption that \(u\) is nonzero. Therefore, the (linear) operator \(\mathcal{L}_g\) is injective on its domain of definition.
	
	Next, we define its inverse on \(L^2(M)\) spectrally as
	\[
	\left( \mathcal{L}_g \right)^{-1} f = \sum_{k=0}^\infty \frac{1}{(\lambda_k + m)\ln(\lambda_k + m)} \langle \phi_k, f \rangle \phi_k, \quad f \in L^2(M).
	\]
	It is immediate that
	\[
	\mathcal{L}_g \circ \left( \mathcal{L}_g \right)^{-1} f = \sum_{k=0}^\infty \langle \phi_k, f \rangle \phi_k = f, \quad f \in L^2(M),
	\]
	so by setting
	\[
	u = \left( \mathcal{L}_g \right)^{-1} f
	\]
	we obtain the unique solution of \eqref{equation_u} in \(\mathbb{H}(M)\subset L^2(M)\), for any \(f \in L^2(M)\).

	\textbf{Direct Problem:} Let \(\O\) be a non-empty open subset of \(M\), and let \(f \in C^\infty(\O)\). Consider the direct problem associated with the Schrödinger operator
	\begin{equation} \label{Schrodinger_operator_for_IP}
		\mathcal{L}_g u + V u = f \quad \text{in } M,
	\end{equation}
	where \(\mathcal{L}_g\) is the operator defined in \eqref{Laplace_log_laplace_spec_def}, and \(V\) is a smooth potential function on \(M\). Before proceeding further, define the following Hilbert space:
	\begin{align}\label{new_hilbert_space}
		\mathcal{H}(M) := \left\{ u \in L^2(M) \,\middle|\, \sum_{k=0}^\infty \left[(\lambda_k + m)\log(\lambda_k + m)\right] |\pi_k u|^2 < \infty \right\}.
	\end{align}
	Observe that \(\mathcal{H}\) is a Hilbert space with respect to the graph norm $\norm{u}^2_{\mathcal{H}(M)}:=\sum_{k=0}^\infty (\lambda_k + m)\log(\lambda_k + m) \,|\pi_ku |^2$. For given that \(m > 1\), note that $(\lambda_k + m)\leq(\lambda_k + m)\log(\lambda_k + m)\leq(\lambda_k + m)^2 \leq C\,(\lambda_k + m)^2 [\log(\lambda_k + m)]^2$, for some $C>0$
	which implies that
	\begin{align}\label{relation_btw_spaces}
		H^{2+2s}(M)\subset\mathbb{H}(M) \subset H^2(M)\subset \mathcal{H}(M)\subset H^1(M), \quad \text{for every $s>0$}. 
	\end{align}
	
	\begin{proposition}\label{uniq_sol}
		Let $V \in C^{\infty}(M)$ and let $m > 1$.  Let $0$ is not an eigenvalue of $(-\Delta_g+m)\log(-\Delta_g+m)+V$. Then for any $f \in C^\infty(M)$, the equation \eqref{Schrodinger_operator_for_IP} admits a unique solution $u^f \in C^\infty(M)$.
	\end{proposition}
	
	\begin{proof}
		We start with demonstrating that the operator $ (-\Delta_g+m)\log(-\Delta_g+m)+ 1 : \mathcal{H}(M) \to \mathcal{H}'(M)$ is invertible, where $\mathcal{H}'(M)$ is the dual of $\mathcal{H}(M)$. We define the bilinear form \(B(u,v) \) by
		\begin{align}\label{bilinear_form}
			B : \mathcal{H}(M)& \times \mathcal{H}(M) \longrightarrow \mathbb{C},\notag \\
			(u, v) &\longmapsto \langle \mathcal{L}_g u, v \rangle_{L^2(M)} + \langle  u, v \rangle_{L^2(M)}.
		\end{align}
		Applying the Cauchy–Schwarz inequality to $B(\cdot,\cdot)$ yields
		\begin{align}
			|B(u, v)| 
			&\leq \|u\|_{\mathcal{H}(M)} \|v\|_{\mathcal{H}(M)} +  \|u\|_{L^2(M)} \|v\|_{L^2(M)}\notag \\
			&\leq 2 \|u\|_{\mathcal{H}(M)} \|v\|_{\mathcal{H}(M)},
		\end{align}
		which implies that the bilinear form $B(\cdot,\cdot)$ is continuous in $\mathcal{H}(M)$.Moreover, the bilinear form is coercive on \(\mathcal{H}(M)\) since
		\begin{align}\label{coercive}
			B(u,u)&=\sum_{k=0}^\infty \left[(\lambda_k + m)\log(\lambda_k + m)\right] \,|\pi_ku|^2 + \int_M |u|^2 \,dV_g\notag\\
			&\geq \sum_{k=0}^\infty \left[(\lambda_k + m)\log(\lambda_k + m)\right] \,|\pi_ku|^2 +  \norm{u}_{L^2(M)}, \notag\\
			&=\sum_{k=0}^\infty \left[(\lambda_k + m)\log(\lambda_k + m)+1\right] \,|\pi_ku|^2\geq C\norm{u}^2_{\mathcal{H}(M)}.
		\end{align}
		By the Lax–Milgram theorem, equation the operator $ (-\Delta_g+m)\log(-\Delta_g+m)+ 1 : \mathcal{H}(M) \to \mathcal{H}'(M)$ is invertible. The operator multiplication by $(V - 1) \in C^\infty(M)$ is compact from $\mathcal{H}(M)$ to $\mathcal{H}'(M)$, see \cite[Theorems 2.3.6 and 2.3.1]{agranovich2015sobolev}. Hence, the operator $ (-\Delta_g+m)\log(-\Delta_g+m)+ V: \mathcal{H}(M)$ to $\mathcal{H}'(M)$ is Fredholm of index zero. The condition $0$ is not an eigenvalue of $ (-\Delta_g+m)\log(-\Delta_g+m)+ V$ yields a unique solution to \eqref{Schrodinger_operator_for_IP} in $\mathcal{H}(M) \subset H^1(M)$ for any \( f \in C^\infty(M) \).
		
		Next, we claim that \(u \in C^\infty(M)\). To see this, choose a local coordinate chart \((U_p; x_1, x_2, \ldots, x_n)\) around \(p \in M\). Then for any $i=1,2,...,n$,
		\begin{align}\label{abc}
			\mathcal{L}_g(\partial_{x_i}u) = \partial_{x_i}\left(\mathcal{L}_g u\right) = \partial_{x_i}f - \partial_{x_i}(V u).
		\end{align}
		Since \(M\) is compact, \(V, f \in C^\infty(M)\), and \(u \in H^1(M)\), it follows that \(\partial_{x_i}u \in H^1(M)\), i.e., \(u \in H^2(M)\). By induction, \(u \in H^s(M)\) for every \(s \in \mathbb{N}\). Thus,
		\[
		u \in \bigcap_{s \in \mathbb{N}} H^s(M) = C^\infty(M).
		\]
	\end{proof}
	
	\begin{remark}
		In the above proposition $m> 1$ plays an important role. It allows us to apply Lax-Milgram lemma as $\log{(\lambda_k+m)}> 0$ for all $k\in\mathbb{N}$.
	\end{remark}
	\subsection*{Equivalent definition of \texorpdfstring{\eqref{log_laplace_spec_def}}{half_power_spec_def} using functional calculus and semigroup}
	
	Let \((e^{-t(-\Delta_g + m)})_{t \geq 0}\) be the strongly continuous heat semigroup on \(L^2(M)\) generated by the operator \((-\Delta_g + m)\), whose infinitesimal generator has domain \(D((-\Delta_g + m)) = H^2(M)\). For any \(v \in L^2(M)\), the action of the heat semigroup is given by the heat kernel \(\tilde{P}(t, x, y)\):
	\begin{align}
		e^{-t\mathcal{L}_g} \, v(x) = \int_M \tilde{P}(t, x, y)\, v(y)\, dV_g(y).
	\end{align}
	The heat kernel \(\tilde{P}(t, x, y)\) admits the following spectral expansion:
	\begin{equation}\label{rela_betw_ker}
		\tilde{P}_g(t, x, y) = \sum_{k=0}^{\infty} e^{-t(\lambda_k + m)} \phi_k(x)\, \phi_k(y) = e^{-mt} P(t, x, y),
	\end{equation}
	where \(P(t, x, y) \in C^\infty((0, \infty) \times M \times M)\) is the heat kernel associated to the semigroup \((e^{t\Delta_g})_{t \geq 0}\). On any closed Riemannian manifold, the following pointwise upper bound for   heat kernel—due to Grigor'yan \cite{grigor97} holds:

	\begin{theorem}
		[Grigor’yan, \cite{grigor97}]\label{het_ker_est}
		Let \( x, y \) be two points on an arbitrary smooth, connected and compact Riemannian manifold \( M \), and let \( t \in (0,\infty) \). Then
		\begin{equation}
			\left| P_g(t,x,y) \right| \leq \frac{C}{t^{n/2}} \, e^{ -\frac{c\,d_g^2(x, y)}{t} },
		\end{equation}
		where \( C > 0 \), \( c > 0 \), and \( d_g(x, y) \) denotes the Riemannian distance between \( x \) and \( y \).
	\end{theorem}

	\begin{lemma}\label{esti_1}
		For the semigroup $e^{-t(-\Delta_g + m\mathbb{I})}, \,t>0$ we have the following estimate for any $v\in L^\infty(M)$
		\begin{align}
			|e^{-t(-\Delta_g + m\mathbb{I})}| \leq e^{-mt}\norm{H_M}_{L^1(M)} \norm{v}_{L^\infty(M)},
		\end{align}
		where, $H_M(z)=e^{-cz}$\footnote{Here \( z \) stands for the mapping \( z: M \times M \times (0,\infty) \to [0,\infty) \) given by \( (x,y,t) \mapsto \frac{d_g(x,y)}{\sqrt{t}} \), where \( d_g \) is the Riemannian metric on \( (M,g) \).}, and $m\neq0$.
	\end{lemma}
	\begin{proof}
		Let \( v \in L^\infty(M) \). Using relation~\eqref{rela_betw_ker}, Theorem~\ref{het_ker_est}, and the self-similarity of \( H_M\left( \frac{d(x,y)}{t^{1/2}} \right) \), we obtain:
		\begin{align}
			|e^{-t(-\Delta_g + m\mathbb{I})} v(x)| 
			&\leq e^{-m t} \int_M \frac{1}{t^{\frac{n}{2}}} H_M\left( \frac{d(x,y)}{t^{1/2}} \right) |v(y)|\, d_g V(y) \notag \\
			&\leq e^{-m t} \|v\|_{L^\infty(M)} \int_M \frac{1}{t^{\frac{n}{2}}} H_M\left( \frac{d(x,y)}{t^{1/2}} \right) d_g V(y) \notag \\
			&\leq e^{-m t} \|v\|_{L^\infty(M)} \|H_M\|_{L^1(M)}.
		\end{align}
		This concludes the proof.
	\end{proof}

	This lemma~\ref{esti_1} guarantees the following mapping property of the heat semigroup:
	\begin{align}\label{mapping_prop_1}
		\| e^{-t(-\Delta_g + m\mathbb{I})} \|_{L^\infty(M) \to L^\infty(M)} \leq e^{-mt} \leq 1.
	\end{align}
	Recall the following identity,\footnote{Let \( I(\lambda) = \int_0^\infty \frac{e^{-t} - e^{-t\lambda}}{t}\, dt \), for \(\lambda > 0\). By differentiating under the integral sign \cite[pp.\ 237]{Rudin64}, we compute \( I'(\lambda) = \frac{1}{\lambda} \). Thus \( I(\lambda) = \log \lambda + c \). Since \( I(1) = 0 \), it follows that \( c = 0 \), so \( I(\lambda) = \log \lambda \).}
	\begin{equation}
		\log \lambda = \int_0^\infty \frac{e^{-t} - e^{-t\lambda}}{t}\, dt, \quad \lambda > 0.
	\end{equation}
	Using the framework of functional calculus (\cite[Ch.\ 31]{LAX02}), we define the operator
	\begin{equation}
		\log(-\Delta_g + m\mathbb{I}) := \int_0^\infty \frac{e^{-t}\mathbb{I} - e^{-t(-\Delta_g + m\mathbb{I})}}{t}\, dt,
	\end{equation}
	and, for any \( v \in C^\infty(M) \), with the commutation property
	\[
	(-\Delta_g + m \mathbb{I}) \circ \log(-\Delta_g + m \mathbb{I}) = \log(-\Delta_g + m \mathbb{I}) \circ (-\Delta_g + m \mathbb{I}),
	\]
	we define
	\begin{equation}\label{log_A_g_s}
		\mathcal{L}_g\, v(x) := (-\Delta_g + m\mathbb{I}) \circ \log(-\Delta_g + m\mathbb{I})\, v(x) = \int_0^\infty \frac{e^{-t} \mathbb{I} - e^{-t(-\Delta_g + m\mathbb{I})}}{t} \, (-\Delta_g + m\mathbb{I}) v(x) \, dt,
	\end{equation}
	for every \( x \in M \).
	
	The following proposition establishes the well-definedness and pointwise formula of \eqref{log_A_g_s}.
	
	\begin{proposition}\label{prop_uniq_sol}
		For every \( u \in C^\infty(M) \) and \( m > 1 \), the operator \(\mathcal{L}_g = \A_g \circ \log\A_g\) admits the pointwise representation
		\begin{align}\label{pointwise_formula}
			\mathcal{L}_g\, u(x)  = \int_0^\infty \frac{e^{-t} \mathbb{I} - e^{-t\A_g}}{t} \, \A_g u(x) \, dt, \quad \forall x \in M,
		\end{align}
		where $\A_g=(-\Delta_g + m\mathbb{I})$.
	\end{proposition}

	\begin{proof}
		
		One may write
		\begin{align}\label{wd_1}
			\int_0^\infty \frac{e^{-t}\mathbb{I}-e^{-tA_g}}{t} \A_g u(x)\, dt
			= \int_0^1 \frac{e^{-t}\mathbb{I}-e^{-tA_g}}{t}\A_g u(x)\, dt
			+ \int_1^\infty \frac{e^{-t}\mathbb{I}-e^{-tA_g}}{t} \A_g u(x)\, dt.
		\end{align}
		
		Let us justify the first integral,
		\begin{align}\label{wd_2}
			\int_0^1 \frac{e^{-t}\mathbb{I}-e^{-tA_g}}{t}\A_g u(x)\, dt
			= \int_0^1 \partial_t\left(u_0(\theta(t),x) - u(\theta(t),x)\right)\,dt,
		\end{align}
		where \( u(t,x) = e^{-t\mathcal{A}_g}\A_g u(x) \) and \( u_0(t,x) = e^{-t}\A_g u(x) \) for \( (t, x) \in (0, \infty)\times M \). The point \( \theta(t) \in (0, t) \) arises as an intermediate value by the classical mean value theorem. And by applying Lemma~\ref{esti_1} in \eqref{wd_2} we obtains
		\begin{align}\label{wd_3}
			\int_0^1 \frac{\left(e^{-t}\mathbb{I}-e^{-t\mathcal{A}_g}\right)\A_g u(x)}{t}\, dt
			&= \int_0^1 \left(-e^{-t}\A_g u(x) + e^{-t\mathcal{A}_g} \mathcal{A}_g^2 u(x)\right)(\theta(t), x)\,dt \notag \\
			&\leq \left(\left\|\A_g u\right\|_{L^\infty} + \|H_M\|_{L^1(M)} \|\mathcal{A}_g^2 v\|_{L^\infty(M)}\right) \int_0^1 e^{-\theta(t)} dt \notag \\
			&\leq \left\|\A_g u\right\|_{L^\infty} +  \|H_M\|_{L^1(M)} \|\mathcal{A}_g^2 v\|_{L^\infty(M)}.
		\end{align}
		Next, considering the second integral, one can apply Lemma~\ref{esti_1} to obtain
		\begin{align}\label{wd_4}
			\int_1^\infty \frac{e^{-t}\mathbb{I} - e^{-t\A_g}}{t} \A_g u(x)\, dt
			&\leq \int_1^\infty \frac{\left|e^{-t}\mathbb{I} - e^{-t\A_g}\right|}{t}\, \left|\A_g u(x)\right|\, dt\notag \\
			\leq \norm{\A_gu}_{L^\infty(M)} &\int_1^\infty \frac{e^{-t}}{t} dt +\|H_M\|_{L^1(M)} \|\mathcal{A}_g^2 v\|_{L^\infty(M)}\int_1^\infty \frac{e^{-mt}}{t} dt<+\infty.
		\end{align}
		Combining \eqref{wd_3} and \eqref{wd_4} in \eqref{wd_1}, we conclude the well-definedness of \eqref{pointwise_formula}.
	\end{proof}
	
	\section{Unique Continuation Principle(UCP)}
	
	In this section we devout ourselves to the proof of unique continuation principles
	of the non local logarithmic Schr\"odinger operator.
	
	\begin{proof}[\textbf{Proof of the theorem~\ref{thm_ucp}}]
		Let $\omega\Subset \O$ be a nonempty open subset. Since $C^\infty(M)$ is invariant under  $\A_g$, it follows from our assumption that for any integer $k \geq 0$,
		\begin{align}\label{ucp_1}
			\A_g^k v|_\O = 0 \quad \text{and} \quad \mathcal{L}_g \mathcal{A}_g^k v|_\O = 0.
		\end{align}
		Using the pointwise formula   given in \eqref{pointwise_formula}, together with \eqref{ucp_1}, we obtain
		\begin{align}\label{ucp_1.1}
			\int_0^\infty e^{-t\mathcal{A}_g}(\mathcal{A}_g^{k+1} v(x)) \frac{dt}{t} = 0,
		\end{align}
		for every $x \in \O$ and $k = 0, 1, 2, \ldots$.
		
		Moreover, the map $t \mapsto e^{-t\mathcal{A}_g}(\mathcal{A}_g v)$ belongs to $C^\infty\big((0,\infty), C^\infty(M)\big)$. Also on the domain $D(\mathcal{A}_g) = H^2(M)$, the operators commute, i.e., $e^{-t\mathcal{A}_g}\mathcal{A}_g^k = \mathcal{A}_g^k e^{-t\mathcal{A}_g}$. Therefore, for every $t \geq 0$, $x \in \O$, and $k = 0, 1, 2, \ldots$, we have
		\begin{align}\label{ucp_2}
			\int_0^\infty \partial_t^k\left(e^{-t\mathcal{A}_g}\mathcal{A}_g v(x)\right) \frac{dt}{t^{}} = 0.
		\end{align}
		
		Now let $x \in \omega$ and $t > 0$. For $l = 0, 1, 2, \ldots$, we have the integral representation
		\begin{align}\label{ucp_3}
			\partial_t^l \bigl(e^{-t \A_g} \A_gv\bigr)(x) = \int_{M \setminus \O} \tilde{P}(t, x, y) \bigl(\A_g^{l+1} v\bigr)(y) \, dV_g(y),
		\end{align}
		here $\tilde{P}(t,x,y)$ is the heat kernel associated with $\mathcal{A}_g$. It follows that from \eqref{ucp_3} for any $x\in \omega$ and $l=0,1,2..$ We deduce the following estimate
		\begin{align}\label{ucp_4}
			|e^{-t\A_g}\A_g^{l+1}v(x)|= |\partial^l_t\left(e^{-t\A_g}\A_g\right)(x)|\leq \norm{\tilde{P}(t,x,y)}_{L^\infty(\omega\times M\setminus\O)}\norm{\A_g^{l+1}v}_{L^1(M)},
		\end{align}
		it follows from \eqref{ucp_3} that for any $x \in \omega$ and $l = 0, 1, 2, \ldots$, we have the estimate
		\begin{align}\label{ucp_4b}
			\left|\partial_t^l \bigl(e^{-t \mathcal{A}_g} \mathcal{A}_g v\bigr)(x)\right| \leq \|\tilde{P}(t,x,y)\|_{L^\infty(\omega \times (M \setminus \Omega))} \, \|\mathcal{A}_g^{l+1} v\|_{L^1(M)}.
		\end{align}
		For $t\in(0,1)$, using theorem~\ref{het_ker_est} and the relation \eqref{rela_betw_ker} in \eqref{ucp_4}, we obtain
		\begin{align}\label{ucp_5}
			\left|\partial_t^l \bigl(e^{-t \A_g} \A_g v\bigr)(x)\right| \leq C e^{-\frac{c_1}{t}} \norm{\A_g^{l+1} v}_{L^1(M)},
		\end{align}
		where $l=0,1,2,\ldots$ and the constant $c_1$ depends on $d_g(\overline{\omega},M \setminus \O)$.
		Observed that $e^{-t \L_g}$ is a submarkovian semigroup. Then, by \cite[Theorem 1]{varopoulos1985hardy}, we have the following estimate for $t \in [1, \infty)$:
		\begin{align}\label{ucp_6}
			\norm{e^{-t\A_g}\A_gv}_{L^\infty(M)}\leq\frac{C}{t^\frac{n}{2}}\,\norm{\A_gv}_{L^1(M)},
		\end{align}
		for any $l=0,1,2,\ldots$. Next, performing integration by parts $k$ times on \eqref{ucp_2}, it follows from the decay estimates \eqref{ucp_5} and \eqref{ucp_6} that the boundary terms vanish at infinity and zero. Consequently, for every $x\in \omega$ we obtain the integral identity
		\begin{align}
			\label{ucp_7}\int_0^\infty\left( e^{-t\A_g}\A_gv\right)(x) \frac{dt}{t^{k+ 1}}=0.
		\end{align}
		Changing the variable $t=\frac{1}{s}$, for $k=1,2,\ldots$ we get
		\begin{align}\label{ucp_7b}
			& \int_0^\infty s^{k-1}\,\phi(s) \,ds=0\notag\\
			i.e,\,\,& \int_0^\infty s^{k}\,\phi(s) \,ds=0,\quad \text{for $k=0,1,2,\ldots$},
		\end{align}
		where for $x\in \omega$ and $\phi(s)={\left(e^{-\frac{1}{s}\A_g}\A_gv\right)(x)}$.
		
		Combining the estimate \eqref{ucp_5} for \( t = \frac{1}{s} \in (0,1) \) and the estimate \eqref{ucp_6} for \( t = \frac{1}{s} \geq 1 \), we obtain, for all \( s > 0 \),
		\begin{align}\label{ucp_8a}
			|\phi(s)| \leq C e^{-c s}.
		\end{align}
		Consider the Fourier transform of the function \( 1_{[0, \infty)} \phi \):
		\begin{align}\label{ucp_fourier}
			\mathcal{F}(1_{[0, \infty)} \phi)(\xi) = \int_0^\infty \phi(s) e^{-i \xi s} \, ds.
		\end{align}
		From the exponential decay estimate \eqref{ucp_8a}, it follows that \(\mathcal{F}(1_{[0, \infty)} \phi)(\xi)\) extends to a holomorphic function on the half-plane \(\operatorname{Im}(\xi) >-c\). Moreover, in view of \eqref{ucp_7b}, all derivatives of \(\mathcal{F}(1_{[0, \infty)} \phi)(\xi)\) vanish at \(\xi = 0\). By the identity theorem for holomorphic functions, this implies that \(\mathcal{F}(1_{[0, \infty)} \phi)(\xi) \equiv 0\) in the domain, and consequently,
		\[
		\phi(s) = 0 \quad \text{for all } s > 0.
		\]
		\begin{align}\label{ucp_3.19}
			i.e\,\,(e^{-t\A_g}\A_gv)(x)=0,\quad \forall t\geq0 \,\,\text{and $x\in \omega$}.
		\end{align}
		Taking the limit as $t \to 0$, the semigroup property implies
		\[
		\lim_{t \to 0} \left(e^{-t \A_g} \A_g v\right)(x) = (\A_g v)(x)=0.
		\]
		Since $v|_\O = 0$, it follows that
		\begin{align}\label{ucp_08}
			\bigl((-\Delta_g) v\bigr)(x) = 0, \quad \text{for every } x \in \omega.
		\end{align}
		Note that the condition $v|_\O = 0$ combining  with \eqref{ucp_08}, we conclude that
		\[
		v \equiv 0 \quad \text{on } M.
		\]
	\end{proof}
	
	As an application of the unique continuation principle established in theorem~\ref{thm_ucp}, we proceed to prove the following lemma~\ref{non-zero_inn_with_eigen-fun}, which plays an important role in transforming the Calderón problem into the Gel'fand inverse spectral problem.
	
	\begin{lemma}
		\label{non-zero_inn_with_eigen-fun}
		Let \((M, g)\) be a smooth, closed, and connected Riemannian manifold. Let \(\phi \in C^\infty(M)\) be an eigenfunction of \(-\Delta_g\) corresponding to eigenvalue \(\lambda\). Then
		\begin{align}
			\langle S_{M,g,V}(f), \phi \rangle \neq 0, \quad \text{for some } f \in C_0^\infty(\O),
		\end{align}
		where \(S_{M,g,V}(f)\) is the unique solution associated to \(f \in C_0^\infty(\O)\) for equation~\eqref{Schrodinger_operator_for_IP}.
	\end{lemma}
	\begin{proof}
		To prove the lemma by contradiction, suppose that
		\begin{align}\label{inn_1}
			\langle{S_{M,g,V}\,(f)}\, {\phi}, \rangle = 0, \quad \text{for all}\ f \in C_0^\infty(\mathcal{O}).
		\end{align}
		Consider the equation
		\begin{align}\label{inn_2}
			\A_{g,\overline{V}} u = \L_g  u + \overline{V} u = \phi \quad \text{on } M,
		\end{align}
		which admits a unique solution $u \in C^\infty(M)$, thanks to proposition~\ref{uniq_sol}. Combining \eqref{inn_1} and \eqref{inn_2}, we obtain
		\begin{align}
			0 = \langle{S_{M,g,V}}, \, {\phi}\rangle_{L^2(M)} = \langle{S_{M,g,V\,}}, \, {\,\,\A_{g,\overline{V}} u}\rangle_{L^2(M)} = \langle{f},\, {u}\rangle, \quad \forall f \in C_0^\infty(\mathcal{O}).
		\end{align}
		This implies that
		\begin{align}
			u|_{\mathcal{O}} = 0.
		\end{align}
		Now applying $-\Delta_g$ to \eqref{inn_2}, we get
		\begin{align}
			\label{inn_3}
			\L_g  (-\Delta_g) u + (-\Delta_g)(\overline{V} u) &= \lambda \phi.
		\end{align}
		Multiplying \eqref{inn_2} by $-\lambda$ gives
		\begin{align}
			\label{inn_4}
			-\lambda \,\L_g  u - \lambda \overline{V} u &= -\lambda \phi.
		\end{align}
		Adding \eqref{inn_3} and \eqref{inn_4} yields
		\begin{align}
			\L_g \bigl((-\Delta_g) u - \lambda u\bigr) + \bigl((-\Delta_g) - \lambda\bigr)(\overline{V} u) = 0,
		\end{align}
		that is,
		\begin{align}
			\L_g  v = \bigl((-\Delta_g) - \lambda \bigr)(\overline{V} u), \quad \text{where } v = \bigl((-\Delta_g) - \lambda\bigr) u.
		\end{align}
		Since \(u|_{\mathcal{O}} = 0\), we also have \(v|_{\mathcal{O}} = 0\) and \(\L_g v|_{\mathcal{O}} = 0\). Thanks to the unique continuation principle in theorem~\ref{thm_ucp}, we conclude that
		\[
		v \equiv 0 \quad \text{on } M,
		\]
		i.e.,
		\[
		(-\Delta_g) u - \lambda u = 0 \quad \text{on } M \quad \text{and} \quad u|_{\mathcal{O}} = 0.
		\]
		As $M$ is connected, the unique continuation property for the Laplace--Beltrami operator $\Delta_g$ ensures that $u \equiv 0$ on $M$. Substituting this into \eqref{inn_1}, we deduce
		\[
		\phi \equiv 0 \quad \text{on } M,
		\]
		which contradicts the fact that \(\phi\) is a nontrivial. This completes the proof.
	\end{proof}
	
	\bigskip
	
	\section{Calder\`on problem to Gel'fand problem}\label{sec_reduction_to_spectral_data}
	
	In this section, we develop tools that facilitate the reduction of the Calderón problem to Gel'fand problem. We start with the following theorem and a lemma which will play a key role to understand the composite action of the heat semigroup and the non local logarithmic operator.
	
	\begin{theorem}[\cite{rudin1987real}, p.~371]
		\label{upper_half_holomorphic_extension}
		Let \(\phi(s)\in L^2(0,\infty)\), and suppose that \(|\phi(s)|\leq e^{-2\pi cs}\) for all \(s\geq 0\) and for some \(c>0\). Then the following function,
		\begin{align}\label{holo_function}
			f(z):=\int_0^\infty \phi(s) \,e^{2\pi i z s} \, ds,
		\end{align}
		is holomorphic on \(\D_c=\{z = x + i y : y > -c\}\).
	\end{theorem}
	
	\begin{proof}
		For $\phi\in L^2(0,\infty)$ and $z\in \D_c$, the following estimate establishes, well-definedness of $f(z)$ as defined in~\eqref{holo_function}:
		\begin{align}
			\left| f(z) \right| &\leq \int_0^\infty |\phi(s)| e^{-2\pi\,y\,s} ds \leq\int_0^\infty e^{-2\pi(y+c)s}< \infty.
		\end{align}
		Let \( z \in \D_c \) and let \( (z_n) \) in $\D_c$ be a sequence such that \( z_n \to z \). Then, we estimate the difference:
		\begin{align*}
			\left|  f(z_n) - f(z) \right| & \leq  \int_0^\infty  |\phi(s)| |e^{2\pi i z_n s}- e^{2\pi i z s}| ds  \\
			&\leq e^{-2\pi cs}  \int_0^\infty    |e^{2\pi i z_n s}- e^{2\pi i z s} |^2 ds
		\end{align*}
		If \( \operatorname{Im}(z) +c > \delta > 0 \) and  \( \operatorname{Im}(z_n) +c> \delta \), then
		\[
		|e^{2\pi s (i x_n-y_n-c) }- e^{2\pi s (i x-y-c) } | \leq 4 e^{-2\delta s},
		\]
		now, using the dominated convergence theorem (DCT), we conclude that \( f(z) \) is continuous at \( z \in \D_c\). 
		To verify that $f(z)$ is holomorphic in $\D_c$. Let $T$ be any closed triangle contained in $\D_c$. Consider the contour integral
		\begin{align}
			\int_T f(z) \, dz &= \int_T \int_0^\infty \phi(s)\, e^{2\pi i z s} \, ds \, dz.
		\end{align}
		The integrand is absolutely integrable, allowing us to apply Fubini's theorem to interchange the integrals. In addition the function $z \mapsto e^{2\pi i z s}$ is entire, so the contour integral over the closed triangle vanishes. Hence
		\begin{align*}
			\int_T f(z) \, dz = \int_0^\infty \phi(s) \left( \int_T e^{2\pi i z s} dz \right) ds=0.
		\end{align*}
		By Morera’s theorem, it follows that $f(z)$ is holomorphic in $\D_c$.
	\end{proof}
	\begin{lemma}
		\label{Zero}
		Let \(\phi \in L^2(0,\infty)\) and assume \(|\phi(s)| \leq e^{-cs}\) for all \(s \geq 0\), where \(c > 0\). Let the function \(f(z)\) be defined as in~\eqref{holo_function}. Then \(f(z)\) is holomorphic in \(\mathbb{D}_c\). Furthermore, if for every integer \(k \in \mathbb{N} \cup \{0\}\),
		\begin{align}\label{condition_in_lemma_3.2}
			\int_0^\infty s^k\, \phi(s)\, ds = 0,
		\end{align}
		then \(\phi(s) = 0\) for almost every \(s \in (0, \infty)\).
	\end{lemma}
	
	\begin{proof}
		By Theorem~\ref{upper_half_holomorphic_extension}, we conclude that the function \(f(z)\) is holomorphic in \(\mathbb{D}_c\). From the assumption~\eqref{condition_in_lemma_3.2}, all derivatives of \(f(z)\) at the origin vanish. Therefore, by analytic continuation, \(f(z) \equiv 0\), which implies that \(\phi \equiv 0\).
	\end{proof}
	The above lemma helps us to derive the following proposition which is strategically very important to transform the Calderon problem to Gel'fand inverse spectral data problem.
	\begin{proposition}\label{key_prop_for_all_thm}
		Let $(M_i, g_i)$, $i=1,2$, be smooth, closed, connected Riemannian manifolds of dimension $n \geq 2$. Let $\mathcal{O} \subset M_1 \cap M_2$ be a non-empty, open, connected subset such that $M_i \setminus \overline{\mathcal{O}} \neq \emptyset$ for $i=1,2$, and assume $(\mathcal{O}, g_1) = (\mathcal{O}, g_2) := (\mathcal{O}, g)$. Let $u_i \in C^\infty(M_i)$. Further suppose that
		\[
		u_1|_{\mathcal{O}} = u_2|_{\mathcal{O}}, \qquad \L_{g_1}  u_1|_{\mathcal{O}} = \L_{g_2}  u_2|_{\mathcal{O}}.
		\]
		Then, for every $x \in \mathcal{O}$ and $t > 0$,
		\begin{align}
			\left[e^{-t \A_{g_1}} \L_{g_1}  u_1 - e^{-t \A_{g_2}} \L_{g_2}  u_2\right](t,x) = 0.
		\end{align}
	\end{proposition}
	\begin{proof}
		From our assumption that \( u_1 \) and \( u_2 \) coincide on \(\mathcal{O}\), i.e., \( u_1|_{\mathcal{O}} = u_2|_{\mathcal{O}} \), it follows that for every \( k \in \mathbb{N} \cup \{0\} \),
		\begin{align}
			\label{P_1_0}
			\mathcal{A}_{g_1}^k u_1|_{\mathcal{O}} = \mathcal{A}_{g_2}^k u_2|_{\mathcal{O}}.
		\end{align}
		Using the identification \((\mathcal{O}, g_1) = (\mathcal{O}, g_2) = (\mathcal{O}, g)\) and the equality
		\(
		\mathcal{L}_{g_1} u_1|_{\mathcal{O}} = \mathcal{L}_{g_2} u_2|_{\mathcal{O}},
		\)
		we deduce that
		\begin{align}
			\label{P_1_1}
			\mathcal{L}_{g_1} \mathcal{A}_{g_1}^k u_1|_{\mathcal{O}} = \mathcal{L}_{g_2} \mathcal{A}_{g_2}^k u_2|_{\mathcal{O}}.
		\end{align}
		Applying the pointwise formula from  \eqref{pointwise_formula}, this yields, for any \( x \in \mathcal{O} \),
		\begin{align}\label{P_1_1.1}
			\int_0^\infty \left( e^{-t \mathcal{A}_{g_1}} \mathcal{A}_{g_1}^{k+1} u_1 - e^{-t \mathcal{A}_{g_2}} \mathcal{A}_{g_2}^{k+1} u_2 \right)(t,x) \frac{dt}{t} = 0,
		\end{align}
		for all \( k \in \mathbb{N} \cup \{0\} \) and \( t \geq 0 \). Recall that on the domain of \(\mathcal{A}_{g_i}^k\), the operators commute, i.e, \(e^{-t \mathcal{A}_{g_i}} \mathcal{A}_{g_i}^k = \mathcal{A}_{g_i}^k e^{-t \mathcal{A}_{g_i}}.\)
		In particular,
		\begin{align}
			\label{P_1_2}
			\partial_t^k \left( e^{-t \mathcal{A}_{g_i}} u_i \right) = (-1)^k e^{-t \mathcal{A}_{g_i}} \mathcal{A}_{g_i}^k u_i.
		\end{align}
		Using \eqref{P_1_2} along with \eqref{P_1_1}, for every \( k = 1, 2, \ldots \), we obtain
		\begin{align}
			\label{P_1_2.1}
			\int_0^\infty \partial_t^k \left( e^{-t \mathcal{A}_{g_1}} u_1 - e^{-t \mathcal{A}_{g_2}} u_2 \right)(t,x) \frac{dt}{t} = 0.
		\end{align}
		
		One can observe that for any \( u_i \in C^\infty(M_i) \), the function $\left(e^{-t \mathcal{A}_{g_1}} u_1 - e^{-t \mathcal{A}_{g_2}} u_2 \right) \big|_{(0, \infty) \times \overline{\mathcal{O}}}$
		satisfies the heat equation
		\begin{align*}
			(\partial_t - \mathcal{A}_g) \left(e^{-t \mathcal{A}_{g_1}} u_1 - e^{-t \mathcal{A}_{g_2}} u_2 \right) = 0 \quad \text{in } (0, \infty) \times \mathcal{O},
		\end{align*}
		with initial condition $\left(e^{-t \mathcal{A}_{g_1}} u_1 - e^{-t \mathcal{A}_{g_2}} u_2\right)(0, x) = u_1(x) - u_2(x) = 0,  \,\,\forall x \in \mathcal{O}.$ Moreover, from \eqref{P_1_0}, for any \( k \in \mathbb{N} \),
		\begin{align}\label{P_1_3}
			\left(e^{-t \mathcal{A}_{g_1}} \mathcal{A}_{g_1}^k u_1 - e^{-t \mathcal{A}_{g_2}} \mathcal{A}_{g_2}^k u_2 \right)(t,x) = (-1)^k \partial_t^k \left(e^{-t \mathcal{A}_{g_1}} u_1 - e^{-t \mathcal{A}_{g_2}} u_2 \right)(t,x) = 0.
		\end{align}
		Recall from Theorem~(1) in \cite{varopoulos1985hardy} that the submarkovian semigroup \((e^{-t \mathcal{A}_{g_i}})_{t \geq 0}\) satisfies the  following estimate
		\begin{align}
			\label{P_1_4}
			\| e^{-t \mathcal{A}_{g_i}} u_i \|_{A^\infty(M_i)} \leq C t^{-\frac{n}{2}} \| u_i \|_{L^1(M_i)}.
		\end{align}
		Using \eqref{P_1_4} together with the vanishing property \eqref{P_1_3}, we can apply integration by parts to \eqref{P_1_2.1} without boundary terms. Thus we obtain
		\begin{align}\label{P_1_5}
			\int_0^\infty (e^{-t \mathcal{A}_{g_1}} u_1 - e^{-t \mathcal{A}_{g_2}} u_2)(t,x) \frac{dt}{t^{k+1}} = 0,
		\end{align}
		for all \( x \in \mathcal{O} \) and \( k \in \mathbb{N} \). Now fix \( x \in \mathcal{O} \) and let $g(t) = \left(e^{-t \mathcal{A}_{g_1}} u_1 - e^{-t \mathcal{A}_{g_2}} u_2 \right)(t,x), \quad t \in (0, \infty).$
		From \eqref{P_1_5}, it follows that
		\begin{align}
			\label{P_1_6}
			\int_0^\infty g(t) \frac{dt}{t^{k+1}} = 0.
		\end{align}
		Rewriting \eqref{P_1_6} with the change of variable \( t = \frac{1}{s} \) and defining \( \phi(s) := g\left(\frac{1}{s}\right) \), we obtain
		\begin{align}
			\label{P_1_6.5}
			\int_0^\infty \phi(s)\, s^k\, ds = 0,
		\end{align}
		for all integers \( k = 0, 1, 2, \ldots \).
		\medskip
		
		Here \( g(t), g^\prime(t) \in L^2(0,\infty) \). To see this, note that Lemma~\ref{esti_1} implies
		\begin{align}\label{P_1_6.1}
			\int_0^\infty |g(t)|^2 \, dt \leq C \left( \|u_1\|_{L^\infty(M_1)} \|H_{M_1}\|_{L^1(M_1)} + \|u_2\|_{L^\infty(M_2)} \|H_{M_2}\|_{L^1(M_2)} \right) \int_0^\infty e^{-mt} \, dt < \infty,
		\end{align}
		and\begin{align}\label{P_1_6.1.2}
			\int_0^\infty |g'(t)|^2 \, dt \leq C \left( \|\mathcal{A}_{g_1} u_1\|_{L^\infty(M_1)} \|H_{M_1}\|_{L^1(M_1)} + \|\mathcal{A}_{g_2} u_2\|_{L^\infty(M_2)} \|H_{M_2}\|_{L^1(M_2)} \right) {\small \int_0^\infty e^{-mt} \, dt < \infty.}
		\end{align}

		Now applying the Hardy's inequality\footnote{The celebrated Hardy inequality states that, if $1 < p < \infty$ and if $u$ is a locally
			absolutely continuous function on $(0, \infty)$ with $\liminf_{r\to 0}|u(r)| = 0$, then 
			\[\int_0^\infty \frac{|u(r)|^p}{r^p}\, dr\leq (\frac{p}{p-1})^p\,\int_0^\infty |u^\prime(r)|^p\, dr.\]
			The constant on the right side is the best possible. 
			We refer to \cite{FLW} and reference therein.} and conclude that 
		\begin{align}
			& \int_0^\infty \frac{|g(t)|^2}{t^2}\, dt \leq 4\int_0^\infty |g^\prime (t)|^2\, dt.\notag\\
			i.e,& \,\, \int_0^\infty |g(\frac{1}{s})|^2\, ds= \int_0^\infty |\phi( s)|^2\, ds<\infty.
		\end{align}
		Therefore, $\phi (s)=g(\frac{1}{s})\in L^2(0,\infty)$ and using the lemma~\ref{Zero} we conclude that $\phi (s)\equiv 0$, i.e, $g(t)\equiv 0$ for $t>0$. Indeed, we have
		\begin{align}\label{P_1_7}
			(e^{-t\A_{g_1}}\,u_1-e^{-t\A_{g_2}}\,u_2)(t,x)=0,
		\end{align}
		for every $x\in \O$ and $t\in (0,\infty)$.
		\medskip
		
		Now, fix \( t_0 > 0 \) and \( x \in \mathcal{O} \) and consider the following difference:
		\begin{align*}
			&\left(e^{-t_0\mathcal{A}_{g_1}}\, \mathcal{L}_{g_1}\,u_1 - e^{-t_0\mathcal{A}_{g_2}}\, \mathcal{L}_{g_2}\,u_2\right)(t, x) \\
			&= \left(\mathcal{L}_{g_1}\,e^{-t_0\mathcal{A}_{g_1}}\,u_1 - \mathcal{L}_{g_2}\,e^{-t_0\mathcal{A}_{g_2}}\,u_2\right)(t, x) \\
			&=   \int_0^\infty \left(e^{-t_0\mathcal{A}_{g_1}}\, \mathcal{A}_{g_1}u_1 - e^{-t_0\mathcal{A}_{g_2}}\, \mathcal{A}_{g_2}u_2\right)(t, x) \frac{dt}{t } \\
			&= 0, \qquad \text{(thanks to \eqref{P_1_7})}.
		\end{align*}
		Here in last line we use \eqref{P_1_7} and the fact $(\O,g_1)=(\O,g_2)$. Since \( t_0 > 0 \) was arbitrary, it follows that for any \( t > 0 \),
		\begin{align}
			\left( e^{-t\mathcal{A}_{g_1}}\, \mathcal{L}_{g_1} \,u_1 - e^{-t\mathcal{A}_{g_2}}\, \mathcal{L}_{g_2} \,u_2 \right)(t, x) = 0,
		\end{align}
		for all \( x \in \mathcal{O} \) and \( t \geq 0 \). This concludes the proof.
	\end{proof}
	
	In the next proposition, we establish the equality of the spectral data. More precisely, Proposition~\ref{Spectural data equal prop} demonstrates how the Calderón problem can be reduced to the Gel'fand problem.

	\begin{proposition}\label{Spectural data equal prop}
		Let $V_j\in C^{\infty}(M_j)$ such that zero is not an eigenvalue of $\L_{g_j}+V_j$, defined on $M_j$, for $j=1,2$. Let $u_j:=u^f_j$ be the unique solution of 
		\begin{equation}\label{t1_1}
			[\L_{g_j}+V_j\,]\, u_j^f = f \quad \text{on } M_j,   
		\end{equation}
		where $f\in C^\infty_0(\mathcal{O})$.  Furthermore, assume     
		\begin{align}\label{t1_p3_0}
			\left( e^{-t\mathcal{A}_{g_1}}\, \mathcal{L}_{g_1} \,u_1 - e^{-t\mathcal{A}_{g_2}}\, \mathcal{L}_{g_2} \,u_2 \right)(t, x) = 0,
		\end{align}
		for all $x \in \mathcal{O}$ and $t > 0$. Then the following statements hold:
		\begin{enumerate}
			\item \(\lambda_k^1 = \lambda_k^2 := \lambda_k\) for all \(k \in \mathbb{N}\), and 
			\((\pi_k^1 u_1)(x) = (\pi_k^2 u_2)(x)\) for every \(x \in \mathcal{O}\).
			\item \(\dim \ker(-\Delta_{g_1} - \lambda_k) = \dim \ker(-\Delta_{g_2} - \lambda_k) := d_k\), for all \(k \in \mathbb{N}\).
			\item There exists an orthonormal Schauder basis 
			\[
			\{\psi^{(j)}_{m,k} : k \in \mathbb{N},\ m = 1,\ldots,d_k \}
			\]
			of \(L^2(M_j)\), where for each \(k\), the set 
			\(\{\psi^{(j)}_{m,k}\}_{m=1}^{d_k}\) consists of eigenfunctions of \(-\Delta_{g_j}\) corresponding to \(\lambda_k\), such that
			\[
			\psi^{(1)}_{m,k}\big|_{\mathcal{O}} = \psi^{(2)}_{m,k}\big|_{\mathcal{O}},
			\quad \text{for all } k \in \mathbb{N}, \ m = 1, \ldots, d_k.
			\]
		\end{enumerate}
		
		Here, \(\lambda_k^j\) denotes the \(k\)-th eigenvalue of \(-\Delta_{g_j}\) on \((M_j, g_j)\), 
		\(\pi_k^j u_j\) denotes the orthogonal projection of \(u_j\) onto the eigenspace associated with \(\lambda_k^j\), 
		and for each \(k\), the vectors \(\psi^{(j)}_{m,k}\) form an orthonormal basis of that eigenspace.
	\end{proposition}
	\begin{proof}
		First, observe that constant functions are the only harmonic functions corresponding to the eigenvalue \( \lambda_k = 0 \). Therefore, the spectral data agree for \( \lambda_k = 0 \). To prove the result for nonzero \( \lambda_k \), we expand \eqref{t1_p3_0} using the spectral representation as follows
		{\footnotesize \begin{align}\label{P_2_2}
				{  \sum_{k=1}^\infty e^{-t[\lambda_k^{(1)}+m]}\,[\lambda_k^{(1)}+m]\,\log\left(\lambda_k^{(1)}+m\right)\,(\pi_k^{(1)}u_1)(x)= \sum_{k=1}^\infty e^{-t[\lambda_k^{(2)}+m]}\,[\lambda_k^{(2)}+m]\,\log\left(\lambda_k^{(2)}+m\right)\,(\pi_k^{(2)}u_2)(x),}
		\end{align}}
		here $x\in \O$ and $t>0$.
		
		Next we justify that the both series in \eqref{P_2_2} are uniformly convergent. Indeed, for $i=1,2,$ the series
		\begin{align}
			\label{P_2_3}
			\sum_{k=1}^\infty e^{-t[\lambda_k^{(i)}+m]}\,[\lambda_k^{(i)}+m]\,\log\left(\lambda_k^{(i)}+m\right)\,(\pi_k^{(i)}u_i)(x)
		\end{align}
		converges uniformly for every $x\in M_i$ and $t>0.$
		\newline
		\textbf{Uniform convergency of \eqref{P_2_3}:}
		To establish uniform convergence of the series in \eqref{P_2_3}, recall that for each \( k \), we have
		\begin{align}\label{P_2_3.1}
			\pi^{(1)}_{k}\, u_1=\sum_{l=1}^{d_k^{(1)}} \langle{u_1},\,{\phi_{k,l}^{(1)}}\rangle_{L^2(M_1)}\, \phi_{k,l}^{(1)},
		\end{align}
		here \( \{ \phi_{k,l}^{(1)} \} \) is an orthonormal basis for the eigenspace associated to \( \lambda_k^{(1)} \). Now any integer \( p \geq 1 \), expanding via powers of the operator gives
		\begin{align}\label{P_2_4}
			\pi_k^{(1)} u_1 = \sum_{l=1}^{d_k^{(1)}} \left[ \lambda_k^{(1)} + m \right]^{-p} \langle{\A_{g_1}^p\,u_1}\, {\phi_{k,l}^{(1)}}\rangle_{L^2(M_1)}\, \phi_{k,l}^{(1)}.
		\end{align}
		From this, we obtain the pointwise estimate
		\begin{align}\label{P_2_4}
			&|\pi_k^{(1)} u_1(x)| 
			\leq \left[ (\lambda_k^{(1)}) + m \right]^{-p} \| \A_{g_1}^p u_1 \|_{L^2(M_1)} \sum_{l=1}^{d_k^{(1)}} | \phi_{k,l}^{(1)}(x)|\notag\\i.e\quad &\| \pi_k^{(1)} u_1 \|_{L^\infty(M_1)} \leq \left[ \lambda_k^{(1)} + m \right]^{-p}
			\| \A_{g_1}^p u_1 \|_{L^2(M_1)}
			\sum_{l=1}^{d_k^{(1)}} \| \phi_{k,l}^{(1)} \|_{L^\infty(M_1)}.
		\end{align}
		
		To proceed, recall the following sup-norm estimate for \( L^2 \)-normalized eigenfunctions: there exists a constant \( C > 0 \) such that
		\begin{equation}
			\label{P_2_5}
			\| \phi_{k,l}^{(1)}\|_{L^\infty(M_1)} \le C \, (\lambda_k^{(1)})^{\frac{n-1}{4}}\le C \, (\lambda_k^{(1)}+m)^{\frac{n-1}{4}},
		\end{equation}
		for all \( \lambda_k^{(1)} \ge 1 \); see \cite[Sections 3.2, formula (3.2.2)]{sogge14}. 
		Additionally, a consequence of Weyl’s law asserts that there is a constant \( C>0 \) such that
		\begin{equation}
			\label{P_2_6}
			N(\lambda) \le C \lambda^{\frac{n}{2}}\le C( \lambda+m)^{\frac{n}{2}},
		\end{equation}
		for all sufficiently large \( \lambda \), where \( N(\lambda) \) denotes the number of eigenvalues of \( -\Delta_{g_1} \), counted with multiplicity, that are less than or equal to \( \lambda \); see \cite[Theorem 3.3.1]{sogge14}. From \eqref{P_2_6}, it follows that
		\begin{equation}
			\label{P_2_7}
			d_k^{(1)} \le C\, (\lambda_{k}^{(1)}+m)^{\frac{n}{2}}, \qquad (\lambda_k^{(1)}+m)\ge\lambda_k^{(1)} \ge C^{-\frac{2}{n}} k^{\frac{2}{n}},
		\end{equation}
		for all sufficiently large \( k \). 
		
		\medskip
		
		By applying the results from \eqref{P_2_4}, \eqref{P_2_5}, and \eqref{P_2_6}, we deduce that
		\begin{align}\label{P_2_8}
			\left|(\lambda_k^{(1)} + m )\,\log\left(\lambda_k^{(1)}+m\right)\,(\pi_k^{(1)}\,u_1)(x)\right|\leq &C\,(\lambda_k^{(1)} + m)^2\,\norm{\pi_k^{(1)}\,u_1}_{L^2(M_1)}\notag\\
			&\leq C\,(\lambda_k^{(1)}+m)^{-(p-\frac{3n+7}{4})} \norm{\L^p_{g_1}\,u_1}_{L^2(M_1)}
		\end{align}
		Let us choose $p\in \N$, such that $p-\frac{3n+7}{4}\geq n$. Then \eqref{P_2_8} implies that
		\begin{align}\label{P_2_8.1}
			\left| (\lambda_k^{(1)} + m ) \,\log\left(\lambda_k^{(1)}+m\right)\,(\pi_k^{(1)}\,u_1)(x)\right|\leq C\,k^{-2}\,\norm{\L_{g_1}u_1}_{L^2(M_1)}\\ 
			i.e,\quad  \left|e^{-t[\lambda_k^{(1)} + m ]}\,(\lambda_k^{(1)} + m ) \,\log\left(\lambda_k^{(1)}+m\right)\,(\pi_k^{(1)}\,u_1)(x)\right|\leq C\,k^{-2}\,\norm{\L_{g_1}u_1}_{L^2(M_1)}\label{P_2_8.2}
		\end{align}
		In the last step, we employed Weyl's law, which asserts the asymptotic relation \(\lambda_k^{(1)} \sim C k^{2/n}\) for large \(k\). Consequently, this implies \((\lambda_k^{(1)})^{-n} \sim C k^{-2}\). Therefore, the series
		\[
		\sum_{k=1}^\infty e^{-t \, (\lambda_k^{(1)}+m)} \big(\lambda_k^{(1)}+m\big )\,\log\left(\lambda_k^{(1)}+m\right)\, \, \big(\pi_k^{(1)} u_1 \big)(x)
		\]
		is uniformly convergent for every \( x \in M_1 \). By symmetry, the analogous series
		\[
		\sum_{k=1}^\infty e^{-t \, (\lambda_k^{(2)}+m)} \big(\lambda_k^{(2)}+m\big) \log\left(\lambda_k^{(2)}+m\right)\, \, \big(\pi_k^{(2)} u_2 \big)(x)
		\]
		also converges uniformly for every \( x \in M_2 \) and \( t > 0 \).
		\medskip
		
		Taking the Laplace transform of both sides of \eqref{P_2_2} with respect to \( t \), we obtain
		\begin{align}\label{P_2_9}
			\sum_{k=1}^\infty \frac{(\lambda_k^{(1)} + m ) \,\log\left(\lambda_k^{(1)}+m\right) \big(\pi_k^{(1)} u_1\big)(x)}{\lambda_k^{(1)} + m + z} 
			= \sum_{k=1}^\infty \frac{  \big(\lambda_k^{(2)}+m\big) \log\left(\lambda_k^{(2)}+m\right)\, \big(\pi_k^{(2)} u_2\big)(x)}{\lambda_k^{(2)} + m + z}
		\end{align}
		for all \( x \in \mathcal{O} \) and \( \Re(z) > 0 \).
		
		To proceed further, define $\Omega_i = \mathbb{C} \setminus \left\{ -\big(\lambda_k^{(i)} + m^2\big): k \geq 1 \right\},$
		and for each \( x \in M_i \) and \( \Re(z) > 0 \), set
		\begin{align}\label{P_2_9.1}
			\mathcal{R}^{(i)}(z, x) := \sum_{k=1}^\infty \frac{(\lambda_k^{(i)} + m ) \,\log\left(\lambda_k^{(i)}+m\right) \, (\pi_k^{(i)} u_i)(x)}{\lambda_k^{(i)} + m + z}.
		\end{align}
		We claim that, for every \( x \in M_i \), the function \( z \mapsto \mathcal{R}^{(i)}(z, x) \) is holomorphic on \( \Omega_i \), with simple poles at each point \( z = -\big[\lambda_k^{(i)} + m\big] \).
		
		Observe that for fixed \( x \in M_1 \), each term in the series defining \( \mathcal{R}^{(1)}(z, x) \) is holomorphic on \( \Omega_1 \). Consequently, to ensure that \( \mathcal{R}^{(1)}(z, x) \)  is holomorphic on \( \Omega_1 \), it suffices to show that for each \( x \in M_1 \), the series converges uniformly on every compact subset \( K \subset \Omega_1 \). 
		
		Let $R>0$ such that $K\subset B(0,R)$. Since the eigenvalues are discrete, therefore only finitely many $-\big(\lambda_k^{(1)} + m\big) $ are inside of $\overline{B(0,R)}$. Therefore
		\begin{equation}\label{P_2_10}
			\min_{-(\lambda_k^{(1)} + m)\in \overline{B(0,R)},z\in K} \,\,|\lambda_k^{(1)} + m+z|\,>0,
		\end{equation}
		for all other values \( -\big[\lambda_k^{(1)} + m\big] \in \mathbb{C} \setminus \overline{B(0, R)} \), we have
		\begin{align}
			\label{P_2_11}
			\min_{z \in K} \left| \lambda_k^{(1)} + m + z \right| = \mathrm{dist}\left( \lambda_k^{(1)} + m,\, K \right) \geq \mathrm{dist}\left( \partial B(0, R),\, K \right) > 0.
		\end{align}
		In light of \eqref{P_2_10} and \eqref{P_2_11}, there exists a constant \( c > 0 \) such that
		\begin{align}\label{P_2_12}
			\left| \lambda_k^{(1)} + m + z \right| > c, \quad \forall k \in \mathbb{N}, \quad \forall z \in K.
		\end{align}
		Now, using \eqref{P_2_12} together with estimate \eqref{P_2_8.1}, we conclude that
		\begin{align}\label{P_2_13}
			\frac{\left|(\lambda_k^{(1)} + m ) \,\log\left(\lambda_k^{(1)}+m\right)\, \big(\pi_k^{(1)} u_1\big)(x)\right|}{|\lambda_k^{(1)} + m  +z|} \leq \frac{c}{k^2}\, \norm{\L^p_{g_1}\,u_1}_{L^2(M_1)},
		\end{align}
		for sufficiently large \( k \geq 1 \) and choose \( p \in \mathbb{N} \) such that $p - \frac{3n +7}{4} \geq n.$
		Then the bound given in \eqref{P_2_13} shows that the series \( \mathcal{R}^{(1)}(x,z) \) converges uniformly for every compact subset \( K \subset \Omega_1 \). Consequently, the function $z \mapsto \mathcal{R}^{(1)}(x,z)$
		is holomorphic on \( \Omega_1 \). This concludes the proof of the claim. Next applying analytic continuation on \eqref{P_2_9}, we have for each $x\in \O$
		\begin{align}
			\mathcal{R}^{(1)}(x,z)=\mathcal{R}^{(2)}(x,z) \quad \text{for all } x \in \mathcal{O} \text{ and } z \in \mathbb{C} \setminus \bigcup_{k \in \mathbb{N}} \{ -(\lambda_k^{(1)} + m), -(\lambda_k^{(2)} + m)\}.
		\end{align}
		\textbf{Equality of eigenvalues and projection operator:}
		Let begin with $k=1$, and assume that $\lambda_1^{(1)}\le\lambda_1^{(2)}$, then for $x\in \O$
		\begin{align}\label{P_2_14}
			\left(\lambda_k^{(1)} + m\right)\, \log\left(\lambda_k^{(1)}+m\right) &(\pi^{(1)}_1 u_1)(x)=\lim_{z\to -[\lambda_k^{(1)} + m^2]}\,(z+(\lambda_k^{(1)} + m) \,\mathcal{R}^{(1)}(z, x)\notag \\
			&=\lim_{z\to -[\lambda_k^{(1)} + m]}\,(z+\lambda_k^{(1)} + m) \,\mathcal{R}^2(z, x)\notag\\
			&= 
			\begin{cases} 
				0, & \text{if} \quad \lambda_1^{(1)} \ne \lambda_1^{(2)}, \\
				(\lambda_k^{(1)} + m) \log\left(\lambda_k^{(1)}+m\right) (\pi_1^{(2)}u_2)(x), & \text{if} \quad \lambda_1^{(1)} = \lambda_1^{(2)}.
			\end{cases}
		\end{align}
		Here, we used the ordering $\lambda_1^{(1)} \leq \lambda_1^{(2)} < \lambda_2^{(2)} < \lambda_3^{(2)} < \cdots.$ By Lemma~\ref{non-zero_inn_with_eigen-fun}, there exists a function \( f \in C_0^\infty(\mathcal{O}) \) such that 
		\[
		(u_1^f, \phi_{1,1}^{(1)})_{L^2(M_1)} \neq 0,
		\]
		where \( u_1^f \) satisfies \eqref{t1_1}. Considering the expansion
		\[
		\pi_1^{(1)} u_1^f = \sum_{l=1}^{d_1^{(1)}} (u_1^f, \phi_{1,l}^{(1)})_{L^2(M_1)} \, \phi_{1,l}^{(1)},
		\]
		and noting that \(\phi_{1,1}^{(1)}, \ldots, \phi_{1,d_1^{(1)}}^{(1)}\) are linearly independent on \(\mathcal{O}\), it follows that \(\pi_1^{(1)} u_1^f \not\equiv 0\) on \(\mathcal{O}\). Therefore, from \eqref{P_2_14}, we deduce that $\lambda_1^{(1)} = \lambda_1^{(2)}$, and  $(\pi_1^{(1)} u_1)(x) = (\pi_1^{(2)} u_2)(x)$ for all $x \in \mathcal{O}$.
		In the case when $\lambda_1^{(2)} \leq \lambda_1^{(1)}$, we proceed similarly as above, with the following:
		\begin{align}\label{t1_p3_46}
			\left(\lambda_k^{(2)} + m\right) \log\left(\lambda_k^{(2)}+m\right)& (\pi^{(2)}_1 u_2)(x) = \lim_{z \to -(\lambda_k^{(2)} + m)} (z +  \lambda_k^{(2)}+m) \,\mathcal{R}^{(2)}(z, x) \notag \\
			&=
			\begin{cases} 
				0, & \text{if } \lambda_1^{(1)} \neq \lambda_1^{(2)}, \\[1.5ex]
				\left(\lambda_k^{(1)}+ m\right)  \log\left(\lambda_k^{(1)}+m\right)(\pi_1^{(1)} u_1)(x), & \text{if } \lambda_1^{(1)} = \lambda_1^{(2)}.
			\end{cases}
		\end{align}
		By repeating the same arguments as above, we conclude that
		\[\lambda_1^{(1)} = \lambda_1^{(2)} \quad \text{and} \quad (\pi_1^{(1)} u_1)(x) = (\pi_1^{(2)} u_2)(x) \quad \text{for all } x \in \mathcal{O}.
		\]
		Using induction, one can similarly show that
		\begin{equation}
			\label{t1_p3_47}\lambda_k^{(1)} = \lambda_k^{(2)}:=\lambda_k \quad \text{and} \quad (\pi_k^{(1)} u_1)(x) = (\pi_k^{(2)} u_2)(x) \quad \text{for all } x \in \mathcal{O},\;\; k \in \mathbb{N}.
		\end{equation}
		This establishes the equality of the spectra and the corresponding eigenfunction projections on $\mathcal{O}$ for the two operators.
		
		\medskip
		
		\textbf{Equality of Eigenfunctions:}
		Let $\pi^{(j)}_k u_j \neq 0$. Then, $\pi^{(j)}_k u_j$ is an eigenvector of $-\Delta_{g_j}$ corresponding to the eigenvalue $\lambda_k$, for $j=1,2$. For $j = 1,2$, define
		\begin{align*}
			S_j = &\mathrm{Span}\left\{ \pi^{(j)}_k u_j : u_j := u_j^f\ \text{is the unique solution of~\eqref{t1_1}};\ f \in C_0^\infty(\Omega) \right\} \\
			\implies &S_j\subset \mathrm{Ker}(-\Delta_{g_j} - \lambda_k).
		\end{align*}
		If $\phi^{(1)}_{l,k} \in S_1$ for all $l = 1, \ldots, d_k^{(1)}$, then clearly $S_1 = \mathrm{Ker}(-\Delta_g - \lambda_k)$. Suppose, for contradiction, there exists $\phi^{(1)}_{l_0,k}\notin S_1$ for some $l_0 \in \{1, \ldots, d_k^{(1)}\}$,  then $\phi^{(1)}_{l_0,k} \in S_1^\perp$, and hence
		\[
		\langle u_1^f, \phi^{(1)}_{l_0,k} \rangle_{L^2(M)} = 0, \quad \forall f \in C_0^\infty(\Omega).
		\]
		This contradicts Lemma~\ref{non-zero_inn_with_eigen-fun}. Therefore,
		\begin{equation}\label{t1_p3_48}
			S_1 = \mathrm{Ker}(-\Delta_{g_j} - \lambda_k).
		\end{equation}
		The equality in \eqref{t1_p3_48} and \eqref{t1_p3_47} implies that $S_1$ contains $d_k^{(1)}$ linearly independent eigenvectors, which, when restricted to $\O$, coincide with $d_k^{(1)}$ eigenvectors of $S_2$. Therefore, $d_k^{(1)} \leq d_k^{(2)}.$ Similarly, considering $S_2$, we will get $ d_k^{(1)} \geq d_k^{(2)}.$ Therefore
		\begin{equation}\label{t1_p3_49}
			d_k: =d_k^{(1)} = d_k^{(2)}.
		\end{equation}
		This also shows that for every $k$, there exist eigenvectors $\widetilde{\psi}^{(j)}_{m,k}$ of $-\Delta_{g_j}$ associated with the eigenvalue $\lambda_k$, for $j = 1, 2$ and $m = 1, \ldots, d_k$, such that
		\begin{equation}\label{t1_p3_50}
			\widetilde\psi^{(1)}_{m,k}\big|_{\O} =\widetilde \psi^{(2)}_{m,k}\big|_{\O},
			\quad m = 1, \ldots, d_k.
		\end{equation}
		Thanks to the Gram--Schmidt orthonormalization process, we obtain a set of orthonormal eigenvectors $\{\psi^{(j)}_{m,k}\}_{m=1}^{d_k}$
		of $\Delta_{g_j}$ corresponding to the eigenvalue $\lambda_k$, where
		\begin{equation}\label{t1_p3_51}
			\left\{
			\begin{aligned}
				\psi^{(j)}_{1,k} &= \frac{\widetilde\psi^{(j)}_{1,k}}{\|\widetilde\psi^{(j)}_{1,k}\|_{L^2(M_j)}} \\
				\widetilde{\psi}^{(j)}_{m,k,\text{orth}} &= \widetilde{\psi}^{(j)}_{m,k} - \sum_{p=1}^{m-1} \left\langle \widetilde{\psi}^{(j)}_{m,k}, \psi^{(j)}_{p,k} \right\rangle_{L^2(M_j)} \psi^{(j)}_{p,k} \\
				\psi^{(j)}_{m,k} &= \frac{ \widetilde{\psi}^{(j)}_{m,k,\text{orth}} }{ \left\| \widetilde{\psi}^{(j)}_{m,k,\text{orth}} \right\|_{L^2(M_j)} }
			\end{aligned}
			\right.
		\end{equation}
		In view of \eqref{t1_p3_50} and \eqref{t1_p3_51}, it follows that
		\begin{equation} 
			\psi^{(1)}_{m,k}\big|_{\O } = \psi^{(2)}_{m,k}\big|_{\O},
			\quad m = 1, \ldots, d_k,
		\end{equation}
		and the collection $\{\psi^{(j)}_{m,k} : k \in \mathbb{N},\ m = 1,\ldots,d_k \}$ 
		forms an orthonormal Schauder basis basis of $L^2(M_j)$. 
		
	\end{proof}
	
	\section{Proof of Theorem~\ref{thm_potential_with compact_supp}}\label{pf_0f_recovering_geom_with_supp_potential}
	
	In Section~\ref{sec_reduction_to_spectral_data}, we showed that equality of the Cauchy data sets leads to equality of the spectral information, as stated in Proposition~\ref{Spectural data equal prop}. In this section, banking on the Cauchy data sets and Proposition~\ref{Spectural data equal prop}, we recover the geometric structure of the manifold, as asserted in Theorem~\ref{thm_potential_with compact_supp}.

	\begin{proof}[\textbf{proof of the theorem~\ref{thm_potential_with compact_supp}}]
		Here \( u_j \) satisfy \eqref{t1_1}, and the equality of the Cauchy data sets allows us to apply Proposition~\ref{key_prop_for_all_thm}, which yields:
		\begin{align}\label{T_1_1}
			\left(e^{-t\A_{g_1}} \L_{g_1} \,u_1-e^{-t\A_{g_1}} \L_{g_1} \,u_1\right)(t,x)=0,
		\end{align}
		for $t\geq0$ and all $x\in \O$. using the equation \eqref{t1_1} in the relation \eqref{T_1_1} we obtain
		\begin{align}
			\label{T_1_2}
			&e^{-t\A_{g_1}}\left(f-V_1\,u_1\right)(t,x)-e^{-t\A_{g_2}}\left(f-V_2\,u_2\right)(t,x)=0\notag\\
			\iff &\left( e^{-t\A_{g_1}}f-e^{-t\A_{g_2}}f\right)(t,x)=\left(e^{-t\A_{g_1}}V_1u_1-e^{-t\A_{g_2}}V_2u_2\right)(t,x),
		\end{align}
		for all $x\in \O$ and $t>0$.
		Using spectral expansion, we have the following 
		\begin{align}\label{T_1_3}
			\left(e^{-t\A_{g_1}}V_1u_1-e^{-t\A_{g_2}}V_2u_2\right)(t,x)=\sum_{k=1}^\infty \left[e^{-t\,(\lambda_k^{(1)} + m)}\pi_k^{(1)}(V_1u_1)(x)-e^{-t\,(\lambda_k^{(2)} + m)}\pi_k^{(2)}(V_2u_2)(x)\right].
		\end{align}
		From the proposition~\ref{Spectural data equal prop} we know that $\lambda_k^{(1)})=\lambda_k^{(2)}:=\lambda_k, \forall k\in \N. $ Therefore, relation \eqref{T_1_3} can be rewritten as
		\begin{align}\label{T_1_4}
			\left(e^{-t\A_{g_1}}V_1 u_1 - e^{-t\A_{g_2}}V_2 u_2\right)(t,x) 
			= \sum_{k=1}^\infty e^{-t (\lambda_k + m)} \left[ \pi_k^{(1)}(V_1 u_1)(x) - \pi_k^{(2)}(V_2 u_2)(x) \right],
		\end{align}
		for all \( x \in \mathcal{O} \) and \( t > 0 \).
		
		Next, we claim that for each \( k \in \mathbb{N} \),
		\begin{align}
			\label{T_1_5}
			\pi_k^{(1)}(V_1 u_1)(x) - \pi_k^{(2)}(V_2 u_2)(x) = 0, \quad \forall x \in \mathcal{O}.
		\end{align}
		Thanks to Proposition~\ref{Spectural data equal prop}, we have an orthonormal Schauder basis
		\begin{equation}\label{T_1_6}
			\left\{ \psi^{(j)}_{m,k} : k \in \mathbb{N},\ m = 1,\ldots,d_k \right\}
		\end{equation}
		of \( L^2(M_i) \), where for each \( k \in \mathbb{N} \), the set \(\{\psi^{(j)}_{m,k}\}_{m=1}^{d_k}\) consists of eigenfunctions of \(-\Delta_{g_j}\) corresponding to \(\lambda_k\), and
		\begin{equation}\label{T_1_7}
			\psi^{(1)}_{m,k}\big|_{\mathcal{O}} = \psi^{(2)}_{m,k}\big|_{\mathcal{O}}, \quad \forall\, k \in \mathbb{N},\ m = 1,\ldots,d_k.
		\end{equation}
		Without loss of generality, taking the projection operators \(\pi_k^{(i)}\) with respect to the orthonormal Schauder basis \eqref{T_1_6}, it follows that for any \( x \in \mathcal{O} \) we have
		\begin{align}\label{T_1_9}
			\pi_k^{(1)}(V_1 u_1)(x) - \pi_k^{(2)}(V_2 u_2)(x) =\sum_{m=1}^{d_k} \Big[ 
			\langle V_1 u_1,\, \psi^{(1)}_{m,k} \rangle_{L^2(M_1)}
			- \langle V_2 u_2,\, \psi^{(2)}_{m,k} \rangle_{L^2(M_2)}
			\Big]\, \psi^{(1)}_{m,k}(x).
		\end{align}
		Since $V_1 = V_2 \in C^\infty_c(\O)$ and $u_1|_{\O}=u_2|_{\O}$, so the difference in the inner products becomes
		\begin{align}\label{T_1_10}
			\langle V_1 u_1,\, \psi^{(1)}_{m,k} \rangle_{L^2(M_1)} 
			- \langle V_2 u_2,\, \psi^{(2)}_{m,k} \rangle_{L^2(M_2)} 
			= \int_{\O} V_1 (u_1 - u_2) \, \psi^{(1)}_{m,k} \, dV_g = 0.
		\end{align}
		This concludes the proof of the claim. Moreover, since \( f \in C_0^\infty(\mathcal{O}) \) is arbitrary, combining \eqref{T_1_2} and \eqref{T_1_9} yields
		\begin{equation}\label{T_1_10}
			(e^{-t \A_{g_1}} f)(t,x) = (e^{-t \A_{g_2}} f)(t,x), \quad \forall f \in C_0^\infty(\mathcal{O}), \quad x \in \mathcal{O}, \quad t > 0.
		\end{equation}
		This further implies the equality of the heat kernels i.e.,
		\[
		P_{{g_1}}(t, x, y) = P_{g_2}(t, x, y),\quad\forall \, t>0,\,\,\mbox{and }x,y\in\mathcal{O}.
		\]
		Hence Theorem \ref{hear_ker_diffeo} yields our result.
	\end{proof}
	
	\section{Proof of Theorem~\ref{update_with_smooth_potential} and Theorem~\ref{recovering_metric_and_potential_from_one_manifold}}
	
	In this section, we recover the isometry class of the manifold along with the lower order term, without imposing any additional assumptions on the potential \( V \). In particular, we present the proofs of Theorem~\ref{recovering_metric_and_potential_from_one_manifold} and Theorem~\ref{update_with_smooth_potential}.

	\begin{proof}[\textbf{Proof of theorem~\ref{recovering_metric_and_potential_from_one_manifold}}]
		By the assumption, Proposition~\ref{Spectural data equal prop} applies and provides equality of the   spectral data; combined with Theorem~\ref{Helin}, this implies the claim and completes the proof.
		
	\end{proof}

	\begin{proof}[\textbf{Proof of Theorem~\ref{update_with_smooth_potential}}]
		In view of Proposition~\ref{Spectural data equal prop} and Theorem~\ref{spec_diffeo}, there exists a diffeomorphism
		\[
		\phi : M_1 \to M_2
		\]
		such that \(\phi\big|_{\mathcal{O}} = \mathrm{id}_{\mathcal{O}}\) and 
		\( g_1 = \phi^* g_2 \).
		
		Define \(\widetilde{V}_2 \in C^\infty(M_1)\) by
		\begin{equation}\label{tt_1}
			\widetilde{V}_2(x) := V_2\big( \phi(x) \big), 
			\quad \forall\, x \in M_1.
		\end{equation}
		We claim that
		\[
		V_1 \equiv \widetilde{V}_2 \quad \text{on } M_1.
		\]
		By Lemma~\ref{lem_obstruction}, we have the identity
		\begin{equation}\label{tt_2}
			C^{\mathcal{O}}_{M_2,g_2,V_2}
			= C^{\mathcal{O}}_{M_1, \phi^* g_2, \widetilde{V}_2}
			= C^{\mathcal{O}}_{M_1, g_1, \widetilde{V}_2}.
		\end{equation}
		From the equality of Cauchy data sets in Theorem~\ref{update_with_smooth_potential} together with \eqref{tt_2}, we obtain
		\begin{equation}\label{tt_3}
			C^{\mathcal{O}}_{M_1, g_1, \widetilde{V}_2}
			= C^{\mathcal{O}}_{M_1, g_1, V_1}.
		\end{equation}
		Let \( f \in C_0^\infty(\mathcal{O}) \) be nonzero, and let 
		\(u_1 \in C^\infty(M_1)\) be the unique solution to 
		\begin{equation}\label{tt_4}
			\mathcal{L} _{g_1}u_1 + V_1u_1 = f 
			\quad \text{on } M_1.
		\end{equation}
		Thanks to \eqref{tt_3}, there exists \(u_2 \in C^\infty(M_1)\) such that
		\begin{equation}\label{tt_5}
			\mathcal{L} _{g_1} u_2 + \widetilde{V}_2 u_2 = 0 
			\quad \text{on } \,M_1\setminus \overline{\O},
		\end{equation}
		and, moreover,
		\begin{equation}\label{eq_zero_on_O}
			(u_1 - u_2)\big|_{\mathcal{O}} = 0,
			\quad
			\mathcal{L} _{g_1} (u_1 - u_2)\big|_{\mathcal{O}} = 0.
		\end{equation}
		By the unique continuation property (Theorem~\ref{thm_ucp}), the above relation \eqref{eq_zero_on_O} implies 
		\(u_1 \equiv u_2\) on \(M_1\).  
		Subtracting \eqref{tt_5} from \eqref{tt_4} then gives
		\begin{equation}\label{tt_7}
			\big(\widetilde{V}_2(x) - V_1(x)\big)\,u_1(x) = 0,
			\quad \forall\, x \in M_1\setminus \overline{\O}.
		\end{equation}
		Recall that $V_{1}\big|_{\mathcal{O}} = \widetilde{V}_{2}\big|_{\mathcal{O}}$. Define
		\[
		D = \left\{ x \in M_{1} \setminus \overline{\mathcal{O}} \; : \; u(x) \neq 0 \right\}.
		\]
		To prove the claim, it is enough to show that \(D\) is dense in \(M_{1} \setminus \overline{\mathcal{O}}\).
		\newline
		Suppose, for the sake of contradiction, that there exists a non-empty set \(\omega \subset M_1 \setminus \overline{\mathcal{O}}\) such that \(D \cap \omega = \emptyset\). Then, from \eqref{tt_5}, we have
		\[
		u_1\big|_{\omega} = 0 \quad \text{and} \quad \mathcal{L} _{g_1} u_1\big|_{\omega} = 0.
		\]
		By the unique continuation property (Theorem~\ref{thm_ucp}), it follows that \(u_1 \equiv 0\) on $M_1$. This contradicts our assumption, and thus the claim is proved.
		
	\end{proof}

	\begin{appendix}
		
		\section{Obstruction to uniqueness in the anisotropic Calder\'on problem for non local Schr\"odinger equations }
		
		\label{app_obstruction}
		
		This appendix addresses a non-uniqueness obstruction related to the inverse problem (IP) introduced in the introduction of this article. The result below is included to provide a complete picture and assist the reader. Its proof closely mirrors that of \cite[Lemma A.1]{FKU24}. For related approaches, see also \cite{GU21}.
		\begin{lemma}
			\label{lem_obstruction}
			Let  $m >1$ and $(M_j,g_j)$ be a smooth closed Riemannian manifold of dimension $n\ge 2$ and let $V_j\in C^\infty(M_j)$, $j=1,2$.  Let $O\subset M_1\cap M_2$ be an open nonempty set such that $M_j\setminus\overline{O}\ne 0$, $j=1,2$. Assume that there is a smooth diffeomorphism $\Phi: M_1\to M_2$ such that $g_1= \Phi^\star g_2$,  $\Phi|_{O}=\text{Id}$, and $V_1= V_2\circ\Phi$. Then 
			\begin{equation}
				\label{eq_600_0_lemma}
				\mathcal{C}_{M_2, g_2, V_2}^O = \mathcal{C}_{M_1, g_1, V_1}^O. 
			\end{equation}
		\end{lemma}
		
		\begin{proof}
			Since $\Phi$ is a Riemannian isometry, it satisfies
			\begin{equation}
				(-\Delta_{g_1})(u \circ \Phi) = ((-\Delta_{g_2})u) \circ \Phi,
				\label{eq:iso}
			\end{equation}
			for all $u \in C^\infty(M_2)$; see \cite[pages 99, 100]{craioveanu2013old}.
			This implies that
			\begin{equation}
				(\mathcal{A} _{g_1})(u \circ \Phi) = (\mathcal{A}_{g_2}u) \circ \Phi, \quad \forall u \in C^\infty(M_2).
				\label{eq:sqrtL}
			\end{equation}
			Note that the map
			\begin{align*}
				U\colon L^2(M_2) &\to L^2(M_1) \\
				u &\mapsto u \circ \Phi
			\end{align*}
			is unitary. This follows from~(see \cite[page 78]{craioveanu2013old})
			\begin{align*}
				\|u \circ \Phi\|_{L^2(M_1)}^2 = \int_{M_1} |u \circ \Phi|^2\, dV_{g_1}  = \int_{M_2} |u|^2\, dV_{g_2} = \|u\|_{L^2(M_2)}^2.
			\end{align*}
			Therefore, equation~\eqref{eq:sqrtL} can be rewritten as
			\[
			\mathcal{A} _{g_1} = U \circ \mathcal{A} _{g_2} \circ U^{-1},
			\]
			and using the functional calculus for self-adjoint operators, we conclude
			\begin{equation}
				\mathcal{L}_{g_1} = U \circ \mathcal{L}_{g_2} \circ U^{-1}.\quad  \,\,\,\,\,\text{where} \  \L_g=(-\Delta_g+m)\circ\log(-\Delta_g+m), m>1.
				\label{eq:L_unitary}
			\end{equation}
			Let \(u_2 \in C^\infty(M_2)\) satisfy
			\[
			\mathcal{L} _{g_2} u_2 + V_2 u_2 = 0 \quad \text{on} \quad M_2 \setminus \overline{O}.
			\]
			Using \eqref{eq:sqrtL}, this implies 
			\[
			0 = \mathcal{L}_{g_2} u_2 + V_2 u_2 = \left(\mathcal{L}_{g_1}(u_2 \circ \Phi)\right) \circ \Phi^{-1} + \left(V_1 \circ \Phi^{-1}\right) \left(u_2 \circ \Phi\right) \circ \Phi^{-1} \quad \text{on} \quad M_2 \setminus \overline{O},
			\]
			showing that $u_1 := u_2 \circ \Phi \in C^\infty(M_1)$ and
			satisfies
			\[
			\mathcal{L}_{g_1} u_1 + V_1 u_1 = 0 \quad \text{on} \quad M_1 \setminus \overline{O}.
			\]
			Here, the map \(\Phi : M_1 \setminus \overline{O} \to M_2 \setminus \overline{O}\) is a smooth diffeomorphism and \(\Phi|_{\overline{O}} = \mathrm{Id}\). Impling, the equality of solutions inside the observation set, 
			\[i.e,\quad
			u_2|_O = u_1|_O,
			\]
			together with \eqref{eq:L_unitary}, implies 
			\[
			(\mathcal{L}_{g_2} u_2)|_O = (\mathcal{L}_{g_1} u_1)|_O,
			\]
			showing that \,$\mathcal{C}^O_{M_2,g_2,V_2} \subset \mathcal{C}^O_{M_1,g_1,V_1}.$
			The opposite inclusion can be established by a similar argument. Hence, we have established \eqref{eq_600_0_lemma}.
			
		\end{proof}

	\end{appendix}
	
	\textbf{Acknowledgement:} The authors were funded by the Department of Atomic Energy $($DAE$)$, Government of India.
	
	\vspace{.3cm}
	
	\textbf{Data Availability:} Data sharing is not applicable to this article as no datasets were generated or analyzed during the current study. 
	
	{\small 
		\bibliographystyle{alpha}
		\bibliography{ref}}
\end{document}